\documentclass[12pt]{article}

\usepackage{amssymb,amsthm,amsmath}

\def\Ex{\mathbb E}
\def\E{\mathbb E}
\def\p{\mathbb P}
\def\Pr{\mathbb P}
\def\er{\mathbb R}
\def\ii{\mathbf{i}}

\newcommand*{\ind}[1]{\mathbf{1}_{\{#1\}}}

\def\calE{\mathcal{E}}
\def\bfi{\mathbf{i}}
\def\bfx{\mathbf{x}}
\def\bfy{\mathbf{y}}
\def\lotimes{\mathop{\otimes}\limits}
\newcommand{\R}{\mathbb{R}}
\newcommand{\N}{\mathbb{N}}

\newtheorem{thm}{Theorem}[section]
\newtheorem{lem}[thm]{Lemma}

\newtheorem{prop}[thm]{Proposition}
\newtheorem{cor}[thm]{Corollary}

\theoremstyle{definition}
\newtheorem{rem}{Remark}

\title{Tail and moment estimates for chaoses generated
by symmetric random variables with logarithmically concave tails
\thanks{Research partially supported by MNiSW Grant no. N N201 397437 and the Foundation for Polish Science.}}

\author{Rados{\l}aw Adamczak\thanks{Institute of Mathematics, University of Warsaw, Banacha 2, 02-097 Warszawa, Poland. R.Adamczak@mimuw.edu.pl.},
Rafa{\l} Lata{\l}a \thanks{Institute of Mathematics, University of Warsaw, Banacha 2, 02-097 Warszawa, Poland and Institute of Mathematics, Polish Academy of Sciences, ul. \'{S}niadeckich 8, 00-956 Warszawa, Poland. R.Latala@mimuw.edu.pl.}}

\begin{document}

\maketitle

\begin{abstract}
We present two-sided estimates of moments and tails of polynomial chaoses of order at most three generated by independent symmetric random variables with log-concave tails as well as for chaoses of arbitrary order generated by independent symmetric exponential variables. The estimates involve only deterministic quantities and are optimal up to constants depending only on the order of the chaos variable.\\


2010 AMS Classification: Primary 60E15, Secondary 60G15
\end{abstract}

\section{Introduction}
A (homogeneous) polynomial chaos of order $d$ is a random variable defined as
\begin{align}\label{undecoupled_chaos}
\sum_{i_1,\ldots,i_d = 1}^n a_{i_1,\ldots,i_d}X_{i_1}\cdots X_{i_d},
\end{align}
where $X_1,\ldots,X_n$ is a sequence of independent real random variables and $(a_{i_1,\ldots,i_d})_{1\le i_1,\ldots,i_d\le n}$ is a $d$-indexed symmetric array of real numbers, satisfying $a_{i_1,\ldots,i_d} = 0$ whenever there exists $k\neq l$ such that $i_k = i_l$.

Random variables of this type appear in many branches of modern probability, e.g. as approximations of multiple stochastic integrals, elements of Fourier expansions in harmonic analysis on the discrete cube (when the underlying variables $X_i$'s are independent Rademachers), in subgraph counting problems for random graphs (in this case $X_i$'s are zero-one random variables) or in statistical physics.

Chaoses of order one are just linear combinations of independent random variables and their behavior is  well-understood. Chaoses of higher orders behave in a more complex way as the summands in (\ref{undecoupled_chaos}) are no longer independent. Nevertheless, due to their simple algebraic structure, many counterparts of classical
results for sums of independent random variables are available. Among well known results there are Khinchine type inequalities and tail bounds involving the variance or some suprema of empirical processes (see e.g. \cite{N, G,Bon,AG,Bo} or Chapter 3 of \cite{dlPG}).

In several cases, under additional assumptions on the distribution of $X_i$'s, even more precise results are known, which give two sided estimates on moments of polynomial chaoses in terms of deterministic quantities involving only the coefficients $a_{i_1,\ldots,i_d}$ (the estimates are accurate up to a constant depending only on $d$). Examples include Gaussian chaoses of arbitrary order \cite{L2}, chaoses generated by nonnegative random variables with log-concave tails \cite{LL} and chaoses of order at most two, generated by symmeric radom variables with log-concave tail (\cite{GK} for $d=1$ and \cite{L1} for $d=2$).

The aim of this paper is to provide some extensions of these results. In particular we provide two sided estimates for moments of chaoses of order three generated by symmetric random variables with log-concave tails (Theorems \ref{lower} and \ref{upper}) and for chaoses of arbitrary order, generated by symmetric exponential variables (Theorem \ref{exponential}).

Before we formulate precisely our main results let us recall the notion of decoupled chaos and decoupling inequalities. A decoupled chaos of order $d$ is a random variable of the form
\begin{align}\label{decoupled_chaos}
\sum_{i_1,\ldots,i_d = 1}^n a_{i_1,\ldots,i_d}X_{i_1}^{1}\cdots X_{i_d}^{d},
\end{align}
where $(a_{i_1,\ldots,i_d})_{1\le i_1,\ldots,i_d \le n}$ is a $d$-indexed array of real numbers and $X_{i}^{l}$, $i = 1,\ldots,n$, $l = 1,\ldots,d$, are independent random variables.

One can easily see that each decoupled chaos can be represented in the form (\ref{undecoupled_chaos}) with a modified matrix and for suitably larger $n$. However it turns out that for the purpose of estimating tails or moments of chaoses it is enough to consider decoupled chaoses. More precisely, we have the following important result due to de la Pe\~na and Montgomery-Smith \cite{dlPMS}.

\begin{thm} Let $(a_{i_1,\ldots,i_d})_{1\le i_1,\ldots,i_d \le n}$ be a symmetric $d$-indexed array such that $a_{i_1,\ldots,i_d} = 0$ whenever there exists $k\neq l$ such that $i_k = i_l$. Let $X_1,\ldots,X_{n}$ be independent random variables and $(X_i^{j})_{1\le i \le n}$, $j=1,\ldots,d$, be independent copies of the sequence $(X_i)_{1\le i \le n}$. Then for all $t \ge 0$,
\begin{align*}
L_d^{-1}&\p\Big(\Big|\sum_{i_1,\ldots,i_d = 1}^n a_{i_1,\ldots,i_d}X_{i_1}^{1}\cdots X_{i_d}^{d}\Big|\ge L_d t\Big)\\
&\le
\p\Big(\Big|\sum_{i_1,\ldots,i_d = 1}^n a_{i_1,\ldots,i_d}X_{i_1}\cdots X_{i_d}\Big|\ge t\Big)\\
&\le L_d\p\Big(\Big|\sum_{i_1,\ldots,i_d = 1}^n a_{i_1,\ldots,i_d}X_{i_1}^{1}\cdots X_{i_d}^{d}\Big|\ge L_d^{-1}t\Big),
\end{align*}
where $L_d \in (0,\infty)$ depends only on $d$.
In particular, for all $p \ge 1$,
\begin{align*}
\tilde{L}_d^{-1}&\Big\|\sum_{i_1,\ldots,i_d = 1}^n a_{i_1,\ldots,i_d}X_{i_1}^{1}\cdots X_{i_d}^{d}\Big\|_p
\le \Big\|\sum_{i_1,\ldots,i_d = 1}^n a_{i_1,\ldots,i_d}X_{i_1}\cdots X_{i_d}\Big\|_p\\
&\le \tilde{L}_d \Big\|\sum_{i_1,\ldots,i_d = 1}^n a_{i_1,\ldots,i_d}X_{i_1}^{1}\cdots X_{i_d}^{d}\Big\|_p,
\end{align*}
where $\tilde{L}_d$ depends only on $d$.
\end{thm}

If we are not interested in the values of numerical constants, the above theorem reduces estimation of tails and moments of general chaoses of order $d$ to decoupled chaoses. The importance of this result stems from the fact that the latter can be treated conditionally as chaoses of smaller order, which allows for induction with respect to $d$.
Since the reduction is straightforward, in the sequel when formulating our results we will restrict to the decoupled case.

Let us finish the introduction by remarking that two-sided bounds on moments of chaoses of the form (\ref{undecoupled_chaos}) can be used to give two-sided estimates for more general random variables, i.e. tetrahedral polynomials in $X_1,\ldots,X_d$, e.g. to polynomials in which every variable appears in a power at most $1$. This is thanks to the following simple observation, which to our best knowledge has remained unnoticed.

\begin{prop} \label{tetrahedral_prop} For $j = 0,1,\ldots,d$ let $(a_{i_1,\ldots,i_j}^{j})_{1\le i_1,\ldots,i_j \le n}$ be a $k$-indexed symmetric array
of real numbers (or more generally elements of some normed space), such that $a_{i_1,\ldots,i_j}^{j} = 0$ if $i_k = i_l$ for some $1 \le k < l \le j$ (for $j=0$ we have just a single number $a_\emptyset^{{0}}$). Let $X_1,\ldots,X_n$ be independent mean zero random variables. Then there exists a constant $L_d \in (0,\infty)$, depending only on $d$, such that for all $p \ge 1$,
\begin{displaymath}
\sum_{j=0}^d \Big\|\sum_{i_1,\ldots,i_j=1}^n a_{i_1,\ldots,i_j}^{j}X_{i_1}\cdots X_{i_j}\Big\|_p \le
L_d \Big\|\sum_{j=0}^d\sum_{i_1,\ldots,i_j=1}^n a_{i_1,\ldots,i_j}^{j}X_{i_1}\cdots X_{i_j}\Big\|_p.
\end{displaymath}
\end{prop}

Note that a reverse inequality boils down just to the triangle inequality in $L_p$ and so the above proposition immediately gives two-sided estimates of moments of tetrahedral polynomials from estimates for homogeneous chaoses.
Since the details are straightforward we will not state explicitly the results which can be obtained from the inequalities we present. The easy (given general results on decoupling) but notationally involved proof of Proposition \ref{tetrahedral_prop} is deferred to the appendix.

The organization of the article is as follows. After introducing the necessary notation (Section \ref{Def_not}) we state our main results (Section \ref{Main_results}) and devote the rest of the paper to their quite involved proof.
In the course of the proof we provide entropy estimates for special kinds of metrics on subsets of certain product sets (Section \ref{Entropy_subsection}) as well as bounds on empirical processes indexed by such sets (Section \ref{Section_Gaussian} and Section \ref{partition_section} where we also provide some partition theorems). We believe that these results may be of independent interest. In Section \ref{proof_section} we conclude the proof of our result for chaoses of order three and in Section \ref{proof_section_exp} we give a proof of estimates for chaoses of arbitrary order generated by exponential variables.

\section{Definitions and notation \label{Def_not}}

Let $(X_{i}^{j})_{1\leq i\leq n,1\leq j\leq d}$ be a matrix of independent
symmetric random variables with
logarithmically concave tails, i.e.\ such that the functions
$N_{i}^{j}\colon [0,\infty) \\ \rightarrow[0,\infty]$ defined by
\[
N_{i}^{j}(t)=-\ln \Pr(|X_{i}^{j}|\geq t)
\]
are convex. We assume that r.v.'s are normalized in such a way that
\begin{equation}
\label{norm}
\inf\{t\geq 0\colon N_{i}^{j}(t)\geq 1\}=1.
\end{equation}
We set
\[
\hat{N}_{i}^{j}(t)=\left\{\begin{array}{ll}
t^2 & \mbox{ for }|t|\leq 1
\\
N_i^{j}(|t|) &\mbox{ for }|t|>1.
\end{array}
\right.
\]

\paragraph{Remark} When working with $d=1$ we will suppress the upper index $j$ and write simply $X_i$ or $N_i$.
\medskip

Recall that the $p$-th moment of a real random variable $X$ is defined as $\|X\|_p^p = \E|X|^p$.

For $\bfi\in \{1,\ldots,n\}^d$ and $I\subset\{1,\ldots,d\}$ we write
$\bfi_{I}=(i_{k})_{k\in I}$. By $P_d$ we will denote the family of all partitions
of $\{1,\ldots,d\}$ into nonempty, pairwise disjoint subsets.
For ${\cal J}=\{I_1,\ldots,I_k\}\in P_d$, $p\geq 2$ and a multiindexed
matrix $(a_{\bfi})$ we define
\begin{align*}
&\|(a_{\bfi})\|_{{\cal J},p}^{{\cal N}}
\\
&=
\sum_{s_1\in I_1,\ldots,s_k\in I_k}\sup\Big\{\sum_{\bfi} a_{\bfi}\prod_{l=1}^k x_{\bfi_{I_l}}^l
\colon  \sum_{i_{s_l}}\hat{N}_{i_{s_l}}^{s_l}(\|(x_{\bfi_{I_l}}^{l})_{\bfi_{I_l\setminus\{s_l\}}}\|_2)\leq p,
1\leq l\leq k\Big\}.
\end{align*}

\paragraph{Remark} When $I_l$ is a singletone, i.e. $I_l = \{s_l\}$, then for any fixed value of $i_{s_l}$,
$\|(x_{\bfi_{I_l}}^{l})_{\bfi_{I_l\setminus\{s_l\}}}\|_2 = |x_{i_{s_l}}^l|$.

\medskip
In particular for $d=3$ we have
\begin{align*}
\|(a_{ijk})\|_{\{1,2,3\},p}^{\cal N}=&
\sup
\bigg\{\sum_{ijk}a_{ijk}x_{ijk}\colon
\sum_i \hat{N}_i^{1}\bigg(\sqrt{\sum_{j,k}x_{ijk}^2}\bigg)\leq p\bigg\}
\\
&+\sup
\bigg\{\sum_{ijk}a_{ijk}x_{ijk}\colon
\sum_j \hat{N}_j^{2}\bigg(\sqrt{\sum_{i,k}x_{ijk}^2}\bigg)\leq p\bigg\}
\\
&+\sup
\bigg\{\sum_{ijk}a_{ijk}x_{ijk}\colon
\sum_k \hat{N}_k^{3}\bigg(\sqrt{\sum_{i,j}x_{ijk}^2}\bigg)\leq p\bigg\},
\end{align*}
\begin{align*}
\|(a_{ijk})\|&_{\{1,2\}\{3\},p}^{\cal N}
\\
=&\sup\bigg\{\sum_{ijk}a_{ijk}x_{ij}y_k\colon
\sum_i \hat{N}_i^{1}\bigg(\sqrt{\sum_{j}x_{ij}^2}\bigg)\leq p,
\sum_k\hat{N}_k^{3}(y_k)\leq p\bigg\}
\\
&+\sup\bigg\{\sum_{ijk}a_{ijk}x_{ij}y_k\colon
\sum_j \hat{N}_j^{2}\bigg(\sqrt{\sum_{i}x_{ij}^2}\bigg)\leq p,
\sum_k\hat{N}_k^{3}(y_k)\leq p\bigg\}
\end{align*}
and
\begin{align*}
\|(a&_{ijk})\|_{\{1\}\{2\}\{3\},p}^{\cal N}
\\
&=
\sup\Big\{\sum_{ijk}a_{ijk}x_{i}y_jz_k\colon
\sum_i \hat{N}_i^{1}(x_{i})\leq p,\sum_j \hat{N}_j^{2}(y_{j})\leq p,
\sum_k\hat{N}_k^{3}(z_k)\leq p\Big\}.
\end{align*}

Throughout the article we will write $L_d,L$ to denote constants depending only on $d$ and
universal constants respectively. In all cases the values of a constant may differ
at each occurence.

By $A \sim_d B$ we mean that there exists a constant $L_d \in (0,\infty)$, such that $L_d^{-1}B \le A \le L_d B$.

We will also denote $X^j = (X_i^j)_{1\le i\le n}$ and write $\E_j$ for the expectation with respect to $X^j$.

\section{Main results \label{Main_results}}

\begin{thm}
\label{lower}
For any $d\ge 1$ and $p\geq 2$ we have
\begin{equation}
\label{lowermom}
\Big\|\sum_{\bfi}a_{\bfi}X_{i_1}^{1}\cdots X_{i_d}^{d}\Big\|_p\geq
\frac{1}{L_{d}}\sum_{{\cal J}\in P_d}\|(a_{\bfi})\|_{{\cal J},p}^{{\cal N}}.
\end{equation}
\end{thm}

\begin{thm}
\label{upper}
For $d\leq 3$ and $p\geq 2$,
\begin{equation}
\label{uppermom}
\Big\|\sum_{\bfi}a_{\bfi}X_{i_1}^{1}\cdots X_{i_d}^{d}\Big\|_p\leq
L_{d}\sum_{{\cal J}\in P_d}\|(a_{\bfi})\|_{{\cal J},p}^{{\cal N}}.
\end{equation}
\end{thm}

\begin{rem}
Let $X_i^j=cg_i^j$, where $g_i^j$ are i.i.d. ${\cal N}(0,1)$ r.v's. and $1<c<10/9$ is such
constant that normalization \eqref{norm} holds. Then $t^2/L\leq \hat{N}_{i}^{j}(t)\leq Lt^2$ and
for ${\cal J}=\{I_1,\ldots,I_k\}\in P_d$, $p\geq 2$
\[
\|(a_{\bfi})\|_{{\cal J},p}^{{\cal N}}\sim_d p^{k/2}\|(a_{\bfi})\|_{{\cal J}},
\]
where
\[
\|(a_{\bfi})\|_{{\cal J}}=\sup\Big\{\sum_{\bfi} a_{\bfi}\prod_{l=1}^k x_{\bfi_{I_l}}^l\colon
\|x_{\bfi_{I_l}}^l\|_2\leq 1, 1\leq l\leq k \Big\}.
\]
Theorems \ref{lower} and \ref{upper} (for arbitrary $d$) in this case were established in \cite{L2}.
\end{rem}

A standard application of the Paley-Zygmund inequality (see e.g. Corollary 3.3.2. of \cite{dlPG}) and the fact that $p$-th and $2p$-th moments of chaoses generated by random variables with log-concave tails are comparable up to constants depending only on the order of the chaos yield the following corollary (for details see the proof of Corollary 1 in \cite{L2}).

\begin{cor}
For $d\leq 3$ and $t>0$,
\[
\frac{1}{L_d}e^{-t/L_d}\leq \Pr\Big(\Big|\sum_{\bfi}a_{\bfi}X_{i_1}^{1}\cdots X_{i_d}^{d}\Big|
\geq \sum_{{\cal J}\in P_d}\|(a_{\bfi})\|_{{\cal J},t}^{{\cal N}}\Big)
\leq L_de^{-tL_d}.
\]
\end{cor}

We are not able to show Theorem \ref{upper} for $d>3$ in the general case. However
we know that it holds for exponential random variables.

\begin{thm}\label{exponential}
If $N_i^j(t)=t$ for all $i,j$ and $t>0$ then for any $d$ and $p\geq 2$ the estimate
\eqref{uppermom} holds.
\end{thm}

\section{Proof of Theorem \ref{lower}}

We will proceed by induction with respect to $d$. The case $d=1$ was proved in \cite{GK}. Let us therefore assume the theorem for all positive integers smaller than $d>1$.

Note that since we allow the constants to depend on $d$, it is enough to show that the left-hand side of (\ref{lowermom})
is minorized by each of the summands on the right-hand side.

For any ${\cal J} = \{I_1,\ldots,I_k\} \in P_d$, with $k\ge 2$, the induction assumption applied conditionally on $(X^j)_{j\in I_1}$
gives

\begin{align}\label{lower_est_induction}
&\Big\|\sum_{\bfi}a_{\bfi}X_{i_1}^{1}\cdots X_{i_d}^{d}\Big\|_p \ge \frac{1}{L_{d-\#I_1}}\Big(\E_{I_1}\Big(\Big\|\Big(\sum_{\bfi_{I_1}} a_{\bfi}\prod_{r\in I_1}X_{i_r}^r\Big)_{\bfi_{I_1^c}}\Big\|_{{\cal J}\setminus I_1,p}^{{\cal N'}}\Big)^p\Big)^{1/p},
\end{align}
where ${\cal N}' = (N_i^j)_{1\le i \le n, j\in I_1^c}$.

Let us fix arbitrary $s_1 \in I_1,\ldots, s_k \in I_k$.
We have
\begin{align*}
&\E_{I_1}\Big(\Big\|\Big(\sum_{\bfi_{I_1}} a_{\bfi}\prod_{r\in I_1}X_{i_r}^r\Big)_{\bfi_{I_1^c}}\Big\|_{{\cal J}\setminus I_1,p}^{{\cal N'}}\Big)^p\\
&= \E_{I_1}\Big(\sup\Big\{\Big|\sum_{\bfi_{I_1^c}}\Big(\sum_{\bfi_{I_1}} a_{\bfi}\prod_{r\in I_1}X_{i_r}^r\Big)\prod_{l=2}^k x_{\bfi_{I_l}}^l\Big|\colon \\
&\phantom{aaaaaaaaaaaaaaaaaaaaa}
 \sum_{i_{s_l}}\hat{N}_{i_{s_l}}^l(\|(x_{\bfi_{I_l}}^l)_{\bfi_{I_l\backslash\{s_l\}}}\|_2)\le p, 2\le l\le k\Big\}\Big)^p\\
& \ge \sup\{\E_{I_1}|\sum_{\bfi_{I_1^c}}(\sum_{\bfi_{I_1}} a_{\bfi}\prod_{r\in I_1}X_{i_r}^r)\prod_{l=2}^k x_{\bfi_{I_l}}^l|^p
 \colon\\
 &\phantom{aaaaaaaaaaaaaaaaaaaaa} \sum_{i_{s_l}}\hat{N}_{i_{s_l}}^l(\|(x_{\bfi_{I_l}}^l)_{\bfi_{I_l\backslash\{s_l\}}}\|_2)\le p, 2\le l\le k\}\\
& \ge  \frac{1}{L_{\#I_1}^p}\Big(\sup\{|\sum_{\bfi} a_{\bfi}\prod_{l=1}^k x_{\bfi_{I_l}}^l|
 \colon \sum_{i_{s_l}}\hat{N}_{i_{s_l}}^l(\|(x_{\bfi_{I_l}}^l)_{\bfi_{I_l\backslash\{s_l\}}}\|_2)\le p, 1\le l\le k\}\Big)^p,
\end{align*}
where the last inequality follows from another application of the induction assumption, this time to a chaos of order $\#I_1$.
Since the indices $s_1,\ldots,s_d$ run over sets of cardinality not exceeding $d$, the above estimate together with (\ref{lower_est_induction}) imply that
\begin{displaymath}
\Big\|\sum_{\bfi}a_{\bfi}X_{i_1}^{1}\cdots X_{i_d}^{d}\Big\|_p \ge \frac{1}{L_d}\|a_\bfi\|_{\mathcal{J},p}^{\cal N}.
\end{displaymath}

The case $k = 1$ requires a different approach. Again it is enough to show that for each $l \in \{1,\ldots,d\}$,
\begin{align*}
\Big\|\sum_{\bfi}a_{\bfi}X_{i_1}^{1}\cdots X_{i_d}^{d}\Big\|_p
\ge \sup\Big\{\sum_{\bfi} a_{\bfi} x_{\bfi}
\colon  \sum_{i_{l}}\hat{N}_{i_{l}}^{l}\Big(\Big(\sum_{\bfi_{\{l\}^c}} x_\bfi^2\Big)^{1/2}\Big)\leq p\Big\}.
\end{align*}
Consider any $x_\bfi$ such that $\sum_{i_l}\hat{N}_{i_l}^l((\sum_{\bfi_{\{l\}^c}}x_\bfi^2)^{1/2})\le p$. By the symmetry of $X_i^l$ we have
\begin{align*}
\Big\|\sum_{\bfi}a_{\bfi}X_{i_1}^{1}\cdots X_{i_d}^{d}\Big\|_p^p & = \E_l \E_{\{l\}^c}\Big|\sum_{i_l} X_{i_l}^l \Big|\sum_{\bfi_{\{l\}^c}} a_{\bfi}\prod_{k\neq l} X_{i_k}^k\Big|\Big|^p\\
&\ge \E_l \Big|\sum_{i_l} X_{i_l}^l \E_{\{l\}^c}\Big|\sum_{\bfi_{\{l\}^c}} a_{\bfi}\prod_{k\neq l} X_{i_k}^k\Big|\Big|^p\\
&\ge \frac{1}{L_d^p}\E_l \Big|\sum_{i_l} X_{i_l}^l \Big(\sum_{\bfi_{\{l\}^c}} a_\bfi^2\Big)^{1/2}\Big|^p\\
&\ge \frac{1}{L_d^p}\Big|\sup\Big\{\sum_{i_l} \Big(\sum_{\bfi_{\{l\}^c}} a_\bfi^2\Big)^{1/2}\alpha_{i_l}\colon \sum_{i_l}\hat{N}_{i_l}^l(\alpha_{i_l})\le p\Big\}\Big|^p\\
&\ge \frac{1}{L_d^p}\Big|\sum_{i_l} \Big(\sum_{\bfi_{\{l\}^c}} a_\bfi^2\Big)^{1/2}\Big(\sum_{\bfi_{\{l\}^c}} x_\bfi^2\Big)^{1/2}\Big|^p\\
&\ge \frac{1}{L_d^p}\Big|\sum_\bfi a_\bfi x_\bfi\Big|^p,
\end{align*}
where the first inequality follows from Jensen's inequality, the second one from hypercontractivity of chaoses generated by log-concave random variables combined with the contraction principle and the third one from the induction assumption.

\section{Preliminary facts \label{Preliminary_facts}}
In this section we present the basic notation and tools to be used in the proof of our main results.
\subsection{Some additional notation}
\begin{enumerate}
\item By $\gamma_{n,t}$ we will
denote the distribution of $tG_{n}$, where
$G_{n}=(g_1,\ldots,g_n)$ is the standard Gaussian vector in $\R^n$.

\item By $\nu_{n,t}$ we will
denote the distribution of $t{\cal E}_{n}$, where
${\cal E}_{n}=(\xi_1,\ldots,\xi_n)$ is a random vector in $\er^n$ with independent coordinates
distributed according to the symmetric exponential distribution with parameter 1. Thus
$\nu_{n,t}$ has the density
\[
d\nu_{n,t}(x)=(2t)^{-n}\exp\Big(-\frac{1}{t}\sum_{i=1}^n|x_i|\Big)dx.
\]
We also put ${\cal E}_{n}^{i}=(\xi_1^{i},\ldots,\xi_n^{i})$ for i.i.d.\ copies
of ${\cal E}_{n}$.

Let us note that $\E \xi_i^2 = 2$.

\item For any norm $\alpha$ on $\R^{n_1\cdots n_d} = \R^{n_1}\otimes \cdots\otimes  R^{n_d}$ (which we will identify with the space of $d$-indexed matrices), let $\rho_\alpha$ be the distance on $\R^{n_1}\times \cdots \times \R^{n_d}$, defined by
\begin{displaymath}
\rho_\alpha(\mathbf{x},\mathbf{y}) = \alpha(x_1\otimes\cdots\otimes x_d-y_1\otimes\cdots\otimes y_d),
\end{displaymath}
where $\mathbf{x} = (x_1,\ldots,x_d), \mathbf{y}= (y_1,\ldots,y_d)$.
\medskip

For $\mathbf{x} \in \R^{n_1+\cdots +n_d}$ and $r \ge 0$ let $B_\alpha(\mathbf{x},r)$ be the closed ball in the metric $\rho_\alpha$ with center $\mathbf{x}$ and radius $r$.
\item Now, for $T \subset \R^{n_1}\times\cdots\times \R^{n_d}$, $t > 0$, define
\begin{displaymath}
W_d^T(\alpha,t) = \sum_{k=1}^d t^k\sum_{I \subset \{1,\ldots,d\}, \#I = k}W_I^T(\alpha),
\end{displaymath}
where for $I \subset \{1,\ldots,d\}$,
\begin{displaymath}
W_I^T(\alpha) = \sup_{\mathbf{x}\in T}\Ex\alpha\Big(\Big(\prod_{k\notin I} x_{i_k}^k\prod_{k\in I}g_{i_k}^{k}\Big)_{i_1,\ldots,i_d}\Big).
\end{displaymath}

\item Similarly, for $t>0$, $T\subset \er^{n_{1}}\times\ldots\times\er^{n_{d}}$ we put
\[
  V_{d}^{T}(\alpha,t):=\sum_{k=1}^{d}t^{k}
  \sum_{I\subset\{1,\ldots,d\}\colon \# I=k}V_{I}^{T}(\alpha),
\]
where
\[
  V_{I}^{T}(\alpha):=\sup_{\bfx\in T}\Ex
  \alpha\Big(\Big(\prod_{k\notin I}x_{i_{k}}^{k}\prod_{k\in I}
  \xi_{i_{k}}^{k}\Big)_{i_{1},\ldots,i_{d}}\Big).
\]
\item For $s,t > 0$, $T \subset \R^{n_1}\times\ldots\times \R^{n_d}$, we define
\begin{displaymath}
U_{d}^T(\alpha,s, t) := \sum_{k=1}^d \sum_{I,J\subset \{1,\ldots,d\},\atop\#(I\cup J) = k, I\cap J = \emptyset} s^{\#I}t^{\#J} U_{I,J}^T(\alpha)
\end{displaymath}
where
\begin{displaymath}
U_{I,J}^T(\alpha) := \sup_{\bfx \in T}\E \alpha\Big(\Big(\prod_{k\notin (I\cup J)}x_{i_k}^k\prod_{k\in I} g_{i_k}^{k}\prod_{k\in J}\xi_{i_k}^{k}\Big)_{i_1,\ldots,i_d}\Big).
\end{displaymath}
\end{enumerate}

\paragraph{Remark} Let us notice that $U_{\emptyset,I}^T(\alpha) = V_I^T(\alpha)$, whereas $U_{I,\emptyset}^T(\alpha) = W_I^T(\alpha)$.

The quantity $W_I^T$ was defined in \cite{L2}, where it played an important role in the analysis of moments of Gaussian chaoses. The quantities $V_I^T$ and $U_{I,J}^T$ will play an analogous role for chaoses generated by general random variables with logarithmically concave tails (as will become clear in the next section, they will allow us to bound the covering numbers for more general sets than those which were important in the Gaussian case).

\subsection{Entropy estimates \label{Entropy_subsection}}
In this section we present some general entropy estimates which will be crucial for bounding suprema of stochastic processes in the proof of Theorem \ref{upper}.

The first lemma we will need is a reformulation of Lemma 1 in \cite{L2}. The original statement from \cite{L2} is slightly weaker however the proof given therein justifies the version presented below.

\begin{lem}\label{estmeasgauss}

For any norms $\alpha_{1},\alpha_{2}$ on $\er^{n}$, $y\in B_2^{n}$
and $t>0$,
\[  \gamma_{n,t}\Big(x\colon \alpha_{i}(x-y)\leq 4t\Ex\alpha_{i}(G_{n}), \; i= 1,2\Big)
  \geq \frac{1}{2}e^{-1/(2t^2)}.
\]
\end{lem}

\begin{lem}
\label{estmeasexp}
For any norms $\alpha_{1},\alpha_{2}$ on $\er^{n}$, $y\in  aB_{1}^{n}$
and $t>0$,
\[  \nu_{n,t}\Big(x\colon \alpha_{i}(x-y)\leq 4t\Ex\alpha_{i}(\calE_{n}), \; i= 1,2\Big)
  \geq \frac{1}{2}e^{-a/t}.
\]
\end{lem}

\begin{proof} Let
\[
  K:=\{x\in\er^{n}\colon \alpha_{1}(x)\leq 4t\Ex\alpha_{1}(\calE_{n}),
  \alpha_{2}(x)\leq 4t\Ex\alpha_{2}(\calE_{n})\}.
\]
By Chebyshev's inequality,
\[
  1-\nu_{n,t}(K)\leq
  \Pr(\alpha_{1}(t\calE_{n})> 4\Ex\alpha_{1}(t\calE_{n}))+
  \Pr(\alpha_{2}(t\calE_{n})> 4\Ex\alpha_{2}(t\calE_{n}))< 1/2.
\]
We get for any $y\in B_{1}^{n}$,
\begin{align*}
  \nu_{n,t}(y+K)&=
  (2t)^{-n}\int_{K}\exp\Big(-\frac{1}{t}\sum_{i=1}^n|x_i+y_i|\Big)dx
\\
  &\geq \exp\Big(-\frac{1}{t}\sum_{i=1}^n|y_i|\Big)\int_{K}d\nu_{n,t}(x)
\\
  &\geq
  \exp(-a/t)\nu_{n,t}(K)\geq \frac{1}{2}\exp(-a/t).
\end{align*}
Finally, notice that if $x\in y+K$, then
$\alpha_{i}(x-y)\leq 4t\Ex\alpha_{i}(\calE_{n})$, $i = 1,2$.
\end{proof}

Before we formulate the next lemma, let us define $\mu_{n,s,t}$ (where $s,t > 0$) as the convolution of $\gamma_{n,s}$ and $\nu_{n,t}$.

\begin{lem}\label{estmeas}
For any norms $\alpha_{1},\alpha_{2}$ on $\er^{n}$, any $a > 0$, $y\in B_2^n + aB_{1}^{n}$
and $s,t>0$, let
\begin{align*}
K = \{x\colon &\alpha_{1}(x-y)\leq 4s\Ex\alpha_{1}(G_{n})+4t\Ex\alpha_1(\calE_n),\\
  &\alpha_{2}(x)\leq 4s\Ex\alpha_{2}(G_{n})+4t\Ex\alpha_2(\calE_n)+
  \alpha_{2}(y)\}.
\end{align*}
Then
\begin{align*}
 \mu_{n,s,t}(K) \ge \frac{1}{4}e^{-1/(2s^2) - a/t}.
\end{align*}
\end{lem}

\begin{proof}
We have $y = y_1 + y_2$ for some $y_1 \in B_2^n$, $y_2 \in aB_1^n$. Define
\begin{align*}
K_1 &= \{x\in \R^n\colon \alpha_i(x-y_1) \le 4s\Ex\alpha_i(G_n), i = 1,2\},\\
K_2 & = \{x\in \R^n\colon \alpha_i(x-y_2) \le 4t\Ex\alpha_i(\calE_n), i =1,2\}.
\end{align*}
For $x = x_1 + x_2$, where $x_j \in K_j$, $j=1,2$,
\begin{displaymath}
\alpha_1(x - y) \le \alpha_1(x_1 - y_1) + \alpha_1(x_2-y_2) \le 4s\Ex\alpha_1(G_n) + 4t\Ex\alpha_1(\calE_n)
\end{displaymath}
and similarly
\begin{displaymath}
\alpha_2(x) \le \alpha_2(x - y) + \alpha_2(y) \le 4s\Ex\alpha_2(G_n) + 4t\Ex\alpha_2(\calE_n) + \alpha_2(y),
\end{displaymath}
therefore $K_1 + K_2 \subset K$. We thus have
\begin{displaymath}
\mu_{n,s,t}(K) \ge \mu_{n,s,t}(K_1 + K_2) \ge \gamma_{n,s}(K_1)\nu_{n,t}(K_2) \ge \frac{1}{4}e^{-1/(2s^2) - a/t},
\end{displaymath}
where in the last inequality we used Lemmas \ref{estmeasgauss} and \ref{estmeasexp}.
\end{proof}

\begin{lem}
\label{estmeas2+d}
For any $s,t>0$, $a = (a_1,\ldots,a_d) \in (0,\infty)^d$ and $\bfx\in (B_2^{n_1} + a_1B_1^{n_1})\times \ldots\times (B_2^{n_d} + a_dB_1^{n_d})$
we have

\begin{equation}
\label{estmeasin}
  \mu_{n_{1}+\ldots+n_{d},s,t}\big(B_{\alpha}\big(\bfx,
  U_{d}^{\{\bfx\}}(\alpha,4s,4t)\big)\big)
  \geq 4^{-d}\exp\Big(-\frac{1}{2}ds^{-2} - \|a\|_1t^{-1}\Big).
\end{equation}
\end{lem}

\begin{proof}
We will proceed by induction on $d$. For $d=1$, inequality
(\ref{estmeasin}) follows by Lemma \ref{estmeas}.
Now suppose that (\ref{estmeasin}) holds for $d-1$. We will show that
it is also satisfied for $d$.
Let us first notice that
\begin{equation}
\label{trian}
  \alpha\Big(\lotimes_{i=1}^{d} x^{i}-\lotimes_{i=1}^{d}y^{i}\Big)
  \leq \alpha^{1}(x^{d}-y^{d})+
  \alpha_{y^{d}}
  \Big(\lotimes_{i=1}^{d-1} x^{i}-\lotimes_{i=1}^{d-1}y^{i}\Big),
\end{equation}
where $\alpha^{1}$ and $\alpha_{y}$
are norms on $\er^{n_{d}}$ and $\er^{n_{1}\cdots n_{d-1}}$ respectively,
defined by
\[
  \alpha^{1}(z):=
  \alpha\Big(\lotimes_{i=1}^{d-1}x^{i}\otimes z\Big)\quad
  \mathrm{and}\quad
  \alpha_{y}(z):=\alpha(z\otimes y).
\]
Then
\begin{equation}
\label{norm_1}
  s\Ex\alpha^{1}(G_n) + t\Ex\alpha^1(\calE_n)= sU_{\{d\},\emptyset}^{\{\bfx\}}(\alpha) + tU_{\emptyset, \{d\}}^{\{\bfx\}}(\alpha).
\end{equation}
Moreover if we put $\pi(\bfx)=(x^{1},\ldots,x^{d-1})$ and
define a norm $\alpha_{s,t}^{2}$ on $\er^{n_{d}}$ by
the formula
\[
  \alpha_{s,t}^{2}(y):=U_{d-1}^{\{\pi(\bfx)\}}(\alpha_{y},s,t)
\]
then
\begin{align}
\label{norm_2}
  &s\Ex\alpha_{s,t}^{2}(G_n)+t\Ex\alpha_{s,t}^{2}(\calE_n) + \alpha_{s,t}^{2}(x^{d})\\
  &=
  \sum_{I,J\subset \{1,\ldots,d\}\atop I\cup J\neq \emptyset, I\cap J = \emptyset}
  s^{\#I}t^{\# J}U_{I,J}^{\{\bfx\}}(\alpha,t) - \Big[sU_{\{d\},\emptyset}^{\{\bfx\}}(\alpha) + tU_{\emptyset, \{d\}}^{\{\bfx\}}(\alpha)\Big].\nonumber
\end{align}
Notice also that by the induction assumption we have for any
$z\in \er^{n_{d}}$,
\begin{align}
\label{indass}
  \mu_{n_{1}+\ldots+n_{d-1},s,t}\Big(&\bfy\in \er^{n_{1}+\ldots+n_{d-1}}\colon
  \alpha_{z}
  \Big(\lotimes_{i=1}^{d-1} x^{i}-\lotimes_{i=1}^{d-1}y^{i}\Big)
  \leq \alpha^{2}_{4s,4t}(z)\Big)
\\
\notag
  &\geq 4^{1-d}\exp(-(d-1)s^{-2}/2 - (a_1+\ldots+a_{d-1})t^{-1}).
\end{align}

Finally let
\begin{align*}
  A(\bfx):=\Big\{&\bfy\in \er^{n_{1}+\ldots+n_{d}}\colon
  \alpha^{1}(x^{d}-y^ {d})
  \leq 4s\Ex\alpha^{1}(G_{n_{d}})+4t\Ex\alpha^{1}(\calE_{n_{d}}),
\\
  &\alpha^{2}_{4s,4t}(y^{d})\leq 4s\Ex\alpha^{2}_{4s,4t}(G_{n_{d}})+4t\Ex\alpha^{2}_{4s,4t}(\calE_{n_{d}}) +
  \alpha^{2}_{4s,4t}(x^{d}),
\\
  &\alpha_{y^{d}}
  \Big(\lotimes_{i=1}^{d-1} x^{i}-\lotimes_{i=1}^{d-1}y^{i}\Big)
  \leq \alpha^{2}_{4s,4t}(y^{d})
  \Big\}.
\end{align*}
By (\ref{trian})-(\ref{norm_2}) we get
$A(\bfx)\subset B_{\alpha}(\bfx, U_{d}^{\{\bfx\}}(\alpha,4s,4t))$ and therefore
by (\ref{indass}), Lemma \ref{estmeas} and Fubini's theorem we
get
\begin{align*}
  &\mu_{n_{1}+\ldots+n_{d},s,t}\Big(B_{\alpha}\Big(\bfx,
  U_{d}^{\{\bfx\}}(\alpha, 4s,4t)\Big)\Big) \\
  &\geq \mu_{n_{1}+\ldots+n_{d},s,t}(A(\bfx))\\
  &\geq
  4^{1-d}\exp(-(d-1)s^{-2}/2 - (a_1+\ldots+a_{d-1})t^{-1})\cdot 4^{-1}\exp(- s^{-2}/2 - a_dt^{-1})\\
&= 4^{-d}\exp(-ds^{-2}/2 - \|a\|_1 t^{-1}).
\end{align*}
\end{proof}

\begin{cor}
\label{estentrgen}
For any
$T\subset (B_2^{n_1} + a_1B_1^{n_1})\times \ldots \times(B_2^{n_d}+a_d B_1^{n_d})$ and $s,t\in (0,1]$,
\[
  N\Big(T,\rho_{\alpha},U_{d}^{T}(\alpha,s,t)\Big)\leq
  \exp(Ld s^{-2} + L\|a\|_1 t^{-1}).
\]
\end{cor}

\begin{proof}
Obviously $U_{d}^{T}(\alpha,s,t)\geq
\sup_{\bfx\in T}U_{d}^{\{\bfx\}}(\alpha,s,t)$. Therefore by Lemma
\ref{estmeas2+d}
we have for any $\bfx\in T$,
\begin{equation}
\label{estmeasgen}
  \mu_{n_{1}+\ldots+n_{d},s,t}\Big(B_{\alpha}\Big(\bfx,
  U_{d}^{T}(\alpha,4s,4t)\Big)\Big)
  \geq 4^{-d}\exp(-ds^{-2}/2 - \|a\|_1t^{-1}).
\end{equation}
Suppose that there exist $\bfx_{1},\ldots,\bfx_{N}\in T$ such that
$\rho_{\alpha}(\bfx_{i},\bfx_{j})> U_{d}^{T}(\alpha,s,t)\geq
2U_{d}^{T}(\alpha,s/2,t/2)$ for $i\neq j$. Then sets
$B_{\alpha}(\bfx_{i},U_{d}^{T}(\alpha,s/2,t/2))$ are disjoint, so by
(\ref{estmeasgen}) we obtain $N\leq 4^{d}\exp(32ds^{-2}+ 8\|a\|_1t^{-1})$. Hence
\[
  N\big(T,\rho_{\alpha},U_{d,I}^{T}(\alpha,s,t)\big)\leq 4^{d}\exp(32ds^{-2}+ 8\|a\|_1t^{-1})
  \leq \exp(34ds^{-2}+ 8\|a\|_1t^{-1}).
\]
\end{proof}

We will need the following standard lemma, whose proof we provide for the sake of completeness.

\begin{lem}\label{lem_gaussian_exp_comparison}
For any $n$ and any norm $\alpha$ on $\R^n$, $\E\alpha(G_n) \le 3\E\alpha(\calE_n)$.
\end{lem}
\begin{proof}
Let $g$ and $\xi$ be respectively  standard Gaussian and symmetric exponential random variables. For $t\ge 0$ we have $\p(|g|\ge t) \le e^{-t^2/2}$ and $\p(|\xi|\ge t) = e^{-t}$. Thus for $t \ge 2$ we have $\p(|g|\ge t) \le \p(|\xi|\ge t)$.

Consider now $G_n = (g_1,\ldots,g_n)$, $\calE_n=(\xi_1,\ldots,\xi_n)$. Define moreover independent random variables $X_1,\ldots,X_n$ distributed as $|g|\ind{|g|> 2}$. Since for all $t \ge 0$, $\p(X_i \ge t) \le \p(|\xi_i|\ge t)$ we can assume that $X_i$'s, $g_i$'s and $\xi_i$'s are defined on the same probability space together with a sequence $\varepsilon_1,\ldots,\varepsilon_n$ of independent Rademacher variables, in such a way that for all $i$, $X_i \le |\xi_i|$ pointwise,  $g_i$'s, $\xi_i$'s, $\varepsilon_i$'s are independent and $X_i$'s are independent of $\varepsilon_i$'s. We can write
\begin{align*}
\E\alpha(G_n) =& \E\alpha(\varepsilon_1|g_1|,\ldots,\varepsilon_n|g_n|) \\
\le& \E\alpha(\varepsilon_1|g_1|\ind{|g_1|\le 2},\ldots,\varepsilon_n|g_n|\ind{|g_n|\le 2}) \\
&+ \E\alpha(\varepsilon_1|g_1|\ind{|g_1|> 2},\ldots,\varepsilon_n|g_n|\ind{|g_n|> 2})\\
\le& 2\E\alpha(\varepsilon_1,\ldots,\varepsilon_n) + \E\alpha(\varepsilon_1X_1,\ldots,\varepsilon_nX_n)\\
\le&2\E_\varepsilon \alpha(\varepsilon_1\E_\xi|\xi_1|,\ldots,\varepsilon_n\E_\xi|\xi_n|) + \E\alpha(\varepsilon_1|\xi_1|,\ldots,\varepsilon_n|\xi_n|)\\
\le &3\E\alpha(\xi_1,\ldots,\xi_n),
\end{align*}
where in the second and third inequality we used (conditionally) the contraction principle.
\end{proof}

Corollary \ref{estentrgen} together with Lemma \ref{lem_gaussian_exp_comparison} yield

\begin{cor}\label{estentrgen1}
For any $T \subset (B_2^n + aB_1^n)^d$ and any $t \in (0,1]$,
\begin{displaymath}
N(T,\rho_\alpha,V_d^T(\alpha,t)) \le \exp(L_dt^{-2} + L_dat^{-1}).
\end{displaymath}
\end{cor}

We would like to remark that by applying  Corollary \ref{estentrgen} with $t_ia_i$ instead of $a_i$ and letting $t_i$ tend to 0 or infinity we can obtain similar results for Cartesian products of the form
$\times_{i=1}^d K_i$ where $K_i$ is either $B_2^n$ or $a_i B_1^n$.  Such results can be also obtained directly by following the proof of Corollary \ref{estentrgen} and using Lemmas \ref{estmeasgauss} and \ref{estmeasexp} instead of Lemma \ref{estmeas}. We will need such entropy estimates only for $d=1$ and $K = a B_1^n$. This case, described in the next corollary, follows just from Lemma \ref{estmeasexp}.

\begin{cor}\label{entropy_B_1^n}
For any $a > 0$, $T \subset aB_1^n $ and $t \in (0,1]$,
\begin{displaymath}
N(T,\rho_\alpha,t \E\alpha(\calE)) \le 2 \exp(8at^{-1}).
\end{displaymath}
\end{cor}

\subsection{Concentration of measure for linear combinations of independent random variables with log-concave tails}

Similarly as in \cite{L2}, the proof of our main results will rely on induction with respect to $d$, the order of the chaos variable.
The base of the induction, i.e. the case $d=1$ was obtained in \cite{GK} by Gluskin and Kwapie\'{n} and later extended in \cite{ L3} to linear combinations of independent symmetric random variables with log-concave tails with vector valued coefficients. Below we present the more general vector-valued version, together with some of its rather standard consequences, which provide the toolbox to be used in the proof. All the lemmas below contain the special case of Gaussian variables and reduce in this case to standard facts about the concentration and integrability for suprema of Gaussian processes.

In the rest of this section we will use the assumptions and notation introduced in Section \ref{Def_not} specialized to the case of $d=1$. In particular we will suppress upper indices (see the remark after the definition of the functions $\hat{N}_i^j$).

\begin{lem}[Theorem 1 in \cite{L3}] \label{tail_est_log_conc}
For any bounded set $T \subset \R^n$ and all $p \ge 2$ we have
\begin{align*}
\frac{1}{L}\Big\|\sup_{t \in T} \Big|\sum_{i=1}^n t_i X_i\Big|\Big\|_p \le & \Big\|\sup_{t \in T} \Big|\sum_{i=1}^n t_i X_i\Big|\Big\|_1 \\
&+ \sup \Big\{\sum_{i=1}^n t_ix_i\colon t\in T, x\in \R^n, \sum_{i=1}^n\hat{N}_i(x_i)\le p\Big\} \\
\le& L\Big\|\sup_{t \in T} \Big|\sum_{i=1}^n t_i X_i\Big|\Big\|_p.
\end{align*}
Thus, for any $u >0$,
\begin{align*}
\p\Big(&\sup_{t \in T} \Big|\sum_{i=1}^n t_i X_i\Big| \ge \\
&L\Big[ \Big\|\sup_{t\in T}\Big|\sum_{i=1}^n t_i X_i\Big|\Big\|_1 + \sup \Big\{\sum_{i=1}^n t_ix_i\colon t\in T, x\in \R^n, \sum_{i=1}^n\hat{N}_i(x_i)\le u\Big\}\Big]\Big) \le e^{-u}.
\end{align*}
\end{lem}
\paragraph{Remark} Using the notation of Section \ref{Def_not}, we can write
\begin{displaymath}
\sup \{\sum_{i=1}^n t_ix_i\colon t\in T, x\in \R^n, \sum_{i=1}^n\hat{N}_i(x_i)\le p\} = \sup_{t\in T}\|t\|_{\{1\},p}^{\cal N},
\end{displaymath}
which shows that the above lemma is indeed a strengthening of the case $d=1$ of Theorems \ref{lower} and \ref{upper}.

\begin{lem} \label{union_of_sets}Consider arbitrary sets\ $T_1,\ldots,T_m \subset \R^n$  and let $T=\bigcup_{j=1}^m T_j$. Then
\begin{align*}
\E \sup_{t\in T} \sum_{i=1}^n t_i X_i \le &L\Big(\max_{j\le m}\E\sup_{t\in T_j}\sum_{i=1}^n t_iX_i \\
&+ \sup\{\sum_{i=1}^n (t_i -s_i)x_i\colon t,s\in T,x\in \R^n, \sum_{i=1}^n \hat{N}_i(x_i) \le \log m\}\Big)
\end{align*}
\end{lem}

\begin{proof}

For $m= 1$ the theorem is obvious, so we will assume that $m \ge 2$.
Let us fix arbitrary $s \in T$. Since $\E X_i = 0$, we have
\begin{displaymath}
\E \sup_{t\in T} \sum_{i=1}^n t_i X_i = \E \max_{j\le m} \sup_{t\in T_j} \sum_{i=1}^n (t_i -s_i) X_i \le \E \max_{j\le m} \sup_{t\in T_j} \Big|\sum_{i=1}^n (t_i -s_i) X_i\Big|.
\end{displaymath}
Let $A = \sup\{\sum_{i=1}^n (t_i -s_i)x_i\colon t \in T,x\in \R^n, \sum_{i=1}^n \hat{N}_i(x_i) \le \log m\}$ and note that by the convexity of $N_i$ and the definition of $\hat{N}_i$, for any $u \ge 1$,
\begin{align}\label{no_abs}
\hat{N}_i(x/u) \le \hat{N}_i(x)/u,
\end{align}
which implies that for $u \ge 1$,
\begin{displaymath}
\sup\Big\{\sum_{i=1}^n (t_i -s_i)x_i\colon t \in T,x\in \R^n, \sum_{i=1}^n \hat{N}_i(x_i) \le 2u\log m\Big\} \le 2uA.
\end{displaymath}
Thus by Lemma \ref{tail_est_log_conc} and the union bound, for any $u \ge 1$,
\begin{align*}
\p\Big(\max_{j\le m} &\sup_{t\in T_j} \Big|\sum_{i=1}^n (t_i -s_i) X_i\Big| \ge L\max_j \E \sup_{t\in T_j} \Big|\sum_{i=1}^n (t_i -s_i) X_i\Big| + LuA\Big) \\
&\le m e^{-2u  \log m} \le \frac{1}{2^u},
\end{align*}
which by integration by parts gives
\begin{align*}
\E \sup_{t\in T}\Big| \sum_{i=1}^n (t_i - s_i)X_i \Big| \le &L\Big(\max_{j\le m}\E\sup_{t\in T_j}\Big|\sum_{i=1}^n (t_i-s_i)X_i\Big| + A\Big).
\end{align*}
To finish the proof of the lemma it is therefore sufficient to show that for all $j \le m$,
\begin{align}\label{no_abs_1}
\E \sup_{t\in T_j}\Big| \sum_{i=1}^n (t_i - s_i)X_i \Big| \le L\Big(\E\sup_{t\in T_j}\sum_{i=1}^n t_iX_i + A\Big).
\end{align}
Let us choose any $z \in T_j$. We have
\begin{align}\label{no_abs_2}
\E \sup_{t\in T_j}\Big| \sum_{i=1}^n (t_i - s_i)X_i \Big| &\le \E \sup_{t\in T_j}\Big| \sum_{i=1}^n (t_i - z_i)X_i \Big| +
\E\Big| \sum_{i=1}^n (z_i - s_i)X_i \Big| \nonumber\\
&\le \E \sup_{t\in T_j}\Big| \sum_{i=1}^n (t_i - z_i)X_i \Big| + L\Big(\sum_{i=1}^n (z_i - s_i)^2\Big)^{1/2} \nonumber\\
& \le \E \sup_{t\in T_j}\Big| \sum_{i=1}^n (t_i - z_i)X_i \Big| + LA,
\end{align}
where in the first inequality we used the fact that variances of $X_i$'s are bounded by a universal constants, whereas in the second one, the estimate $(\sum_{i=1}^n (z_i - s_i)^2)^{1/2} = \sup\{\sum_{i=1}^n (z_i-s_i)u_i \colon \sum_{i=1}^n u_i^2 \le 1\} \le (\log 2)^{-1}A$ for $m \ge 2$, which is an easy consequence of (\ref{no_abs}) and the fact that $\hat{N}_i(u) = u^2$ for $|u|\le 1$.

Let us now notice that
\begin{align*}
\E& \sup_{t\in T_j}\Big| \sum_{i=1}^n (t_i - z_i)X_i \Big| = \E\max \Big(\sup_{t\in T_j} (\sum_{i=1}^n (t_i - z_i)X_i)_+, \sup_{t\in T_j} (\sum_{i=1}^n (t_i - z_i)X_i)_-\Big)\\
&\le \E \sup_{t\in T_j} (\sum_{i=1}^n (t_i - z_i)X_i)_+ + \E \sup_{t\in T_j} (\sum_{i=1}^n (t_i - z_i)X_i)_-\\
&= 2 \E \sup_{t\in T_j} (\sum_{i=1}^n (t_i - z_i)X_i)_+=  2\E \sup_{t\in T_j} \sum_{i=1}^n (t_i - z_i)X_i,
\end{align*}
where in the second inequality we used the symmetry of $X_i$'s and in the last one the fact that $z \in T_j$.

The above inequality together with (\ref{no_abs_2}) proves (\ref{no_abs_1}) and ends the proof of the lemma.
\end{proof}

Let us finish this section with a version of Lemma \ref{union_of_sets} in the special case of Gaussian variables. It improves on the inequality of Lemma \ref{union_of_sets}, as it asserts that the constant in front of $\max_j\E\sup_{x\in T_j}\sum_{i=1}^n x_i g_i$ may be taken to be equal to one.
This result is again pretty standard and its proof can be found e.g. in \cite{L2} (see Lemma 3 therein). It is analogous to the argument presented above, but instead of Lemma \ref{tail_est_log_conc} it uses the Gaussian concentration inequality.

\begin{lem}\label{gaussian_concentration}
Let $g_1, \ldots,g_n$ be independent standard Gaussian variables and let $T = \bigcup_{j=1}^m T_j \subset \R^n$. Then
\begin{displaymath}
\E\sup_{x\in T} \sum_{i=1}^n t_i g_i \le \max_{j\le m}\E\sup_{t\in T_j}\sum_{i=1}^n t_ig_i + L\sqrt{\log m}\sup_{s,t\in T} \Big(\sum_{i=1}^n (s_i-t_i)^2\Big)^{1/2}.
\end{displaymath}
\end{lem}

\section{Suprema of some Gaussian processes \label{Section_Gaussian}}

The main result of this section is Proposition \ref{crucial} below, which is a strengthening of Theorem 3 of \cite{L3} in the special case $d=3$. Before stating the proposition we need some additional definitions.

For a triple indexed matrix $A = (a_{ijk})$ and a set $T \subset \R^n \times \R^n$, let us define
\[
\Delta_{A}(T)=
\sup\Big\{\Big(\sum_{k}\Big(\sum_{ij}a_{ijk}(x_iy_j-\tilde{x}_i\tilde{y}_j)\Big)^{2}
\Big)^{1/2}
\colon (x,y),(\tilde{x},\tilde{y})\in T\Big\}
\]
and
\[
s_2^T(A)=\sup_{(x,y)\in T}\Big[\Big(\sum_{jk}\Big(\sum_i a_{ijk}x_i\Big)^{2}\Big)^{1/2}
+\Big(\sum_{ik}\Big(\sum_j a_{ijk}y_j\Big)^{2}\Big)^{1/2}\Big].
\]

\begin{prop}
\label{crucial}
For any $p \ge2 $ and any set $T\subset (B_2^n + \sqrt{p}B_1^n)\times (B_2^n + \sqrt{p}B_1^n)$,
\[
\Ex\sup_{(x,y)\in T}\sum_{ijk}a_{ijk}x_iy_jg_k\leq
L\Big[\sqrt{p}\Delta_{A}(T)+s_2^T(A)+
\frac{1}{\sqrt{p}}\Big(\sum_{ijk}a_{ijk}^2\Big)^{1/2}
\Big].
\]
\end{prop}

Before we pass to the proof of Proposition \ref{crucial} we will prove its counterpart for double-indexed matrices. This simpler result will be used
in the proof of Proposition \ref{crucial}.

\begin{lem}\label{supremum_l1}
For any matrix $B = (b_{ij})_{i,j\le n}$, $a \ge 1$ and $T \subset a B_1^n$,
\begin{align*}
\E\sup_{x \in T} \sum_{ij =1}^n b_{ij}x_i g_j &\le La^{1/2}\|B\|_{\{1,2\}}^{1/2} (\|B\|_{\{1,2\}}\wedge \Delta_B(T))^{1/2} + La^{1/2}\Delta_B(T),\\
&\le L\Big(\|B\|_{\{1,2\}} + a\Delta_B(T)\Big).
\end{align*}
where $\Delta_B(T) = \sup_{x,x' \in T}\Big(\sum_{j=1}^n \Big(\sum_{i=1}^n b_{ij} (x_i-x_i')\Big)^2\Big)^{1/2}$.
\end{lem}

\begin{proof}
Let us consider the process $Z_x = \sum_{i=1}^n b_{ij} x_i g_j$ and the associated metric
\begin{displaymath}
d_Z(x,x') = \|Z_x - Z_{x'}\|_2 = \Big(\sum_{j=1}^n \Big(\sum_{i=1}^n b_{ij} (x_i-x_i')\Big)^2\Big)^{1/2}.
\end{displaymath}

We have $\Delta_B(T) = {\rm diam}_{d_Z} T$.
Since $\E (\sum_{i=1}^n (\sum_{j=1}^n b_{ij} \xi_j)^2)^{1/2} \le \sqrt{2}\|B\|_{\{1,2\}}$,
by Corollary \ref{entropy_B_1^n}, we have for $t \in (0,1]$,
\begin{displaymath}
N(T,d_Z,t\|B\|_{\{1,2\}}) \le \exp(Lat^{-1}),
\end{displaymath}
so for $\varepsilon \le \|B\|_{\{1,2\}}$,
\begin{displaymath}
N(T,d_Z,\varepsilon) \le \exp(L\|B\|_{\{1,2\}}a\varepsilon^{-1}).
\end{displaymath}
By Dudley's bound (see \cite{D} or e.g. Corollary 5.1.6 in \cite{dlPG}) we have
\begin{align*}
\E\sup_{x\in T} Z_x \le& L\int_0^{\Delta_B(T)} \sqrt{\log N(T,d_Z,\varepsilon)}d\varepsilon\\
\le& L\int_0^{\|B\|_{\{1,2\}}\wedge \Delta_B(T)}a^{1/2}\|B\|_{\{1,2\}}^{1/2}\varepsilon^{-1/2}d\varepsilon \\
&+ L\int_{\|B\|_{\{1,2\}}\wedge \Delta_B(T)}^{\Delta_B(T)}a^{1/2}d\varepsilon\\
= & La^{1/2}\|B\|_{\{1,2\}}^{1/2} (\|B\|_{\{1,2\}}\wedge \Delta_B(T))^{1/2} + La^{1/2}\Delta_B(T).
\end{align*}
The second estimate of the lemma follows from the inequality $2\sqrt{xy} \le a^{-1/2}x+a^{1/2}y$.

\end{proof}

\begin{lem}\label{supremum_d2}
For any matrix $B = (b_{ij})_{i,j\le n}$, any $T \subset B_2^n +\sqrt{p}B_1^n$ and $p\ge 1$,
\begin{displaymath}
\E\sup_{x \in T} \sum_{i,j=1}^n b_{ij} x_i g_j \le L\Big( \|B\|_{\{1,2\}} + \sqrt{p}\Delta_B(T)\Big),
\end{displaymath}
where $\Delta_B$ is as in Lemma \ref{supremum_l1}.
\end{lem}

\begin{proof}
Since $\E (\sum_{i=1}^n (\sum_{j=1}^n b_{ij} \xi_j)^2)^{1/2} \le \sqrt{2}\|B\|_{\{1,2\}}$,
by Corollary \ref{entropy_B_1^n} (with $a = \sqrt{p}$ and $t = 1/(\sqrt{2}p)$) there exist sets $K_i \subset \sqrt{p}B_1^n$, $i=1,\ldots,N \le \exp(Lp)$, such that
\begin{displaymath}
\sqrt{p}B_1^n = \bigcup_{i=1}^N K_i
\end{displaymath}
and
\begin{align}\label{supremum_d2_1}
\Delta_B(K_i) \le p^{-1/2}\|B\|_{\{1,2\}}.
\end{align}
By Lemma \ref{gaussian_concentration} we have

\begin{align}\label{supremum_d2_2}
\E\sup_{x\in T} \sum_{ij}b_{ij}x_ig_j &= \E\max_{i\le N}\sup_{x \in T\cap(B_2^n + K_i)} \sum_{ij}b_{ij}x_i g_j\nonumber\\
&\le \max_{i\le N}\E\sup_{x \in T\cap(B_2^n + K_i)} \sum_{ij}b_{ij}x_i g_j + L\sqrt{\log N}\Delta_B(T) \nonumber\\
&\le \max_{i\le N} \Big(\E\sup_{x \in B_2^n} \sum_{ij}b_{ij}x_ig_j + \E\sup_{x \in K_i}\sum_{ij}b_{ij}x_ig_j\Big) + L\sqrt{p}\Delta_B(T)\nonumber\\
&\le \|B\|_{\{1,2\}} + L(\|B\|_{\{1,2\}} + \sqrt{p} \Delta_B(K_i)) + L\sqrt{p}\Delta_B(T),
\end{align}
where in the last inequality we used Lemma \ref{supremum_l1} and the fact that
\begin{displaymath}
\E\sup_{x\in B_2^n}\sum_{ij}b_{ij}x_i g_j = \E \sqrt{\sum_{i}(\sum_j b_{ij} g_j)^2} \le \|B\|_{\{1,2\}}.
\end{displaymath}
Inequalities (\ref{supremum_d2_1}) and (\ref{supremum_d2_2}) imply the lemma.
\end{proof}

For a triple indexed matrix $A = (a_{ijk})_{i,j,k}$, let $\alpha_A$ be a norm on $\R^{n^2}$, given by
\begin{displaymath}
\alpha_A(z) = \Big(\sum_k\Big(\sum_{i,j}a_{ijk}z_{ij}\Big)^2\Big)^{1/2}.
\end{displaymath}
To simplify the notation we will write $\rho_A$ for $\rho_{\alpha_A}$.
Note that
\begin{displaymath}
\rho_A((x_1,y_1),(x_2,y_2)) = (\Ex(X_{(x_1,y_1)} - X_{(x_2,y_2)})^2)^{1/2},
\end{displaymath}
where
\begin{displaymath}
X_{(x,y)} = \sum_{ijk}a_{ijk}x_iy_jg_k.
\end{displaymath}

We will also need a norm on $\R^n\times \R^n$ defined by
\begin{displaymath}
\tilde{\alpha}_A((x,y)) =
\Big(\sum_{jk}\Big(\sum_i a_{ijk}x_i\Big)^2\Big)^{1/2}
+ \Big(\sum_{ik}\Big(\sum_{j}a_{ijk}y_j\Big)^2\Big)^{1/2}.
\end{displaymath}
The corresponding distance on $\R^n\times \R^n$ will be denoted by $\tilde{\rho}_A$.

We will use the following consequences of Corollary \ref{estentrgen1}.

\begin{cor}\label{first_entropy_cor}
For  any $p \ge 1$, any set $T\subset (B_2^n+\sqrt{p}B_1^n)\times (B_2^n+\sqrt{p}B_1^n)$ and any $t \in (0,1]$,
\begin{displaymath}
N(T,\rho_A,t^2\|A\|_{\{1,2,3\}} + ts_2^T(A)) \le \exp(Lt^{-2} +L\sqrt{p}t^{-1}).
\end{displaymath}
\end{cor}

\begin{proof}
It is enough to notice that
\begin{displaymath}
V_{\{1,2\}}^T(\alpha_A) = \Ex(\sum_{k}(\sum_{ij}a_{ijk}\xi_{i}^{1}\xi_j^{2})^2)^{1/2} \le
(\Ex\sum_{k}(\sum_{ij}a_{ijk}\xi_{i}^{1}\xi_j^{2})^2)^{1/2} =2\|A\|_{\{1,2,3\}}
\end{displaymath}
whereas
\begin{align*}
V_{\{1\}}^T(\alpha_A) & + V_{\{2\}}^T(\alpha_A) \\
&= \sup_{(x,y)\in T}\Ex(\sum_{k}(\sum_{ij}a_{ijk} \xi_i y_j)^2)^{1/2} +
\sup_{(x,y)\in T}\Ex(\sum_{k}(\sum_{ij}a_{ijk}x_i \xi_j)^2)^{1/2} \\
&\le \sqrt{2}\sup_{(x,y)\in T} (\sum_{jk}(\sum_i a_{ijk}x_i)^2)^{1/2} + \sqrt{2}\sup_{(x,y) \in T} (\sum_{ik}(\sum_{j}a_{ijk}y_j)^2)^{1/2}\\
&\le 2\sqrt{2}s_2^T(A).
\end{align*}
The statement of the corollary follows now from Corollary \ref{estentrgen1} applied with $d=2$.
\end{proof}

\begin{cor}
\label{second_entropy_cor}
For any $p \ge 1$, any set $T\subset (B_2^n+\sqrt{p}B_1^n)\times (B_2^n+\sqrt{p}B_1^n)$ and any $t \in (0,1]$,
\begin{displaymath}
N(T,\tilde{\rho}_A,t\|A\|_{\{1,2,3\}}) \le \exp(Lt^{-2} +L\sqrt{p}t^{-1}).
\end{displaymath}
\end{cor}

\begin{proof}
Let $(\calE^{1},\calE^{2})$ be a standard exponential random vector with values in $\R^n\times \R^n = \R^{2n}$. We have
\begin{displaymath}
\Ex \tilde{\alpha} (\calE^{1},\calE^{2}) \le 2\sqrt{2}\|A\|_{\{1,2,3\}},
\end{displaymath}
hence the corollary follows from Corollary \ref{estentrgen1} with $d=1$ and the fact that $(B_2^n + \sqrt{p}B_1^n)\times(B_2^n + \sqrt{p}B_1^n) \subset \sqrt{2}B_2^{2n} + 2\sqrt{p}B_1^{2n}$.
\end{proof}

To simplify the formulation of the next lemmas let us denote
\begin{displaymath}
F_A^G(T) = \E\sup_{(x,y)\in T} \sum_{k}\sum_{ij}a_{ijk}x_iy_jg_k.
\end{displaymath}

\begin{lem}\label{first_decomposition_lemma}
For $p \ge 1$ let $(x,y) \in (B_2^n+\sqrt{p}B_1^n)\times (B_2^n+\sqrt{p}B_1^n)$ and let $T \subset (B_2^n+\sqrt{p}B_1^n)\times (B_2^n+\sqrt{p}B_1^n)$. Then, for any $l\ge 0$, there exists a decomposition
\begin{displaymath}
T = \bigcup_{i=l}^N T_l,
\end{displaymath}
with $N \le \exp(L2^{2l}p)$, such that for all $l \le N$ ,

\begin{align}\label{shifted_supremum}
F_A^G((x,y) + T_l) \le F_A^G(T_l) + L\tilde{\alpha}((x,y)).
\end{align}
and
\begin{align}\label{diameter_bound}
\Delta_A(T_l) \le 2^{-l}p^{-1/2}s_2^T(A) + 2^{-2l}p^{-1}\|A\|_{\{1,2,3\}}.
\end{align}
\end{lem}

\begin{proof}
We apply Corollary \ref{first_entropy_cor} with $t = 2^{-l}p^{-1/2}$, which gives us a partition of $T$ into $N \le \exp(L2^{2l}p)$ sets, satisfying
the required diameter bound (\ref{diameter_bound}). Let $B^1 = (b^{1}_{jk})$, $B^2 = (b^{2}_{ik})$ where
\begin{displaymath}
b^{1}_{ik} = \sum_j a_{ijk}y_j, \quad b^{2}_{jk} = \sum_i a_{ijk}x_i
\end{displaymath}

We have
\begin{align*}
\E(\sum_k(\sum_i b^{1}_{ik} \xi_i)^2)^{1/2} +& \E(\sum_k(\sum_j b^{2}_{jk} \xi_j)^2)^{1/2} \\
&\le \sqrt{2}(\sum_{ik}(\sum_j a_{ijk}y_j)^2)^{1/2} + \sqrt{2}(\sum_{jk}(\sum_i a_{ijk}x_i)^2)^{1/2} \\
& = \sqrt{2}\tilde{\alpha}((x,y)),
\end{align*}
therefore by Corollary \ref{estentrgen1} (with $d=1$, $a= \sqrt{p}$ and $t = 1/(L\sqrt{p})$), there exists a partition of $T$ into at most $e^{L p}$ sets $S_l$ such that for all $l$,
\begin{align*}
&\sup_{(x',y'),(x'',y'')\in S_l} \Big[(\sum_k(\sum_i b^{1}_{ik} (x'_i -x''_i))^2)^{1/2} + (\sum_k(\sum_j b^{2}_{jk} (y'_j - y''_j))^2)^{1/2}\Big] \\
&\le \frac{1}{\sqrt{p}}\tilde{\alpha}((x,y)).
\end{align*}

We can intersect this partition with the previous one to obtain a partition of $T$ into at most $e^{C 2^{2l}p}$ sets $T_l$, such that (\ref{diameter_bound}) holds
and the above inequality is satisfied with $T_l$ instead of $S_l$.

Let $\pi_1, \pi_2$ be the projections from $\R^{2n} = \R^n \times \R^n$ onto the first and the second $n$ coordinates respectively and note that
\begin{align}\label{diameter}
&\Delta_{B^1}(\pi_1(T_l)) + \Delta_{B^2}(\pi_2(T_l)) \nonumber\\
&\le 2 \sup_{(x',y'),(x'',y'')\in T} \Big[(\sum_k(\sum_i b^1_{ik} (x'_i - x''_i)^2)^{1/2} + (\sum_k(\sum_j b^2_{jk} (y'_j - y''_j)^2)^{1/2} \Big]\nonumber\\
&\le \frac{2}{\sqrt{p}}\tilde{\alpha}((x,y)).
\end{align}
By the equality $\Ex \sum_{ijk}a_{ijk}x_iy_jg_k = 0$ we get for any $l$,
\begin{align*}
F_A^G((x,y) + T_l) \le& F_A^G(T_l) + \Ex\sup_{(\tilde{x},\tilde{y})\in T_l}\sum_{ijk}a_{ijk}x_i\tilde{y}_j g_k
+ \Ex\sup_{(\tilde{x},\tilde{y})\in T_l}\sum_{ijk}a_{ijk}\tilde{x}_iy_j g_k\\
\le& F_A^G(T_l) + L\Big((\sum_{jk}(\sum_i a_{ijk}x_i)^2)^{1/2} + (\sum_{ik}(\sum_j a_{ijk}y_j)^2)^{1/2}\\
&+\sqrt{p}\Delta_{B^1}(\pi_1(T_l)) + \sqrt{p}\Delta_{B^2}(\pi_2(T_l))\Big)\\
\le& F_A^G(T_l) + L\tilde{\alpha}((x,y)),
\end{align*}
where in the second inequality we used the assumption $T \subset (B_2^n+\sqrt{p}B_1^n)\times (B_2^n+\sqrt{p}B_1^n)$ and Lemma \ref{supremum_d2} (applied to matrices $B^1,B^2$) and
in the last inequality the estimate (\ref{diameter}).
\end{proof}

\begin{lem}\label{second_decomposition_lemma} Let $S$ be a finite subset of $(B_2^n+\sqrt{p}B_1^n)\times (B_2^n+\sqrt{p}B_1^n)$ of cardinality at least 2, such that $S-S \subset (B_2^n+\sqrt{p}B_1^n)\times (B_2^n+\sqrt{p}B_1^n)$. Then, for any $l \ge 0$, there exist finite sets $S_i\subset (B_2^n+\sqrt{p}B_1^n)\times (B_2^n+\sqrt{p}B_1^n)$, and points $(x_i,y_i) \in S_i$ $i =1,\ldots,N$, such that
\begin{itemize}

\item[{\rm (i)}] $2 \le N \le \exp(L2^{2l}p)$,

\item[{\rm(ii)}] $S = \bigcup_{i=1}^N((x_i,y_i)+S_i)$, $S_i - S_i \subset S - S$, $\#S_i \le \#S - 1$,

\item[{\rm(iii)}] $\Delta_A(S_i) \le 2^{-2l}p^{-1}\|A\|_{\{1,2,3\}}$,

\item[{\rm(iv)}] $s_2^{S_i}(A) \le 2^{-l}p^{-1/2}\|A\|_{\{1,2,3\}}$,

\item[{\rm(v)}] $F_A^G((x_i,y_i) + S_i) \le F_A^G(S_i) + Ls_2^S(A)$.
\end{itemize}
\end{lem}

\begin{proof}
Corollary \ref{second_entropy_cor}, applied with $t = 2^{-l-1}p^{-1/2}$, gives us a decomposition
\begin{displaymath}
S = \bigcup_{i=1}^{N_1}((x_i,y_i)+T_i),
\end{displaymath}
where $N_1 \le \exp(L2^{2l}p)$, $(x_i,y_i)\in S$ and $s_2^{T_i}(A) \le 2^{-l-1}p^{-1/2}\|A\|_{\{1,2,3\}}$. Since $\#S \ge 2$ we can assume that $N_1 \ge 2$. We can also assume that the sets $(x_i,y_i) + T_i$ are pairwise disjoint and nonempty, which implies that $\#T_i \le \#S - 1$.

Since $T_i \subset S - (x_i,y_i) \subset (B_2^n+\sqrt{p}B_1^n))\times (B_2^n+\sqrt{p}B_1^n))$, by Lemma \ref{first_decomposition_lemma}, it can be further decomposed into the union
\begin{displaymath}
T_i = \bigcup_{j=1}^{N_2}T_{ij},
\end{displaymath}
with $N_2 \le \exp(L2^{2l}p)$, where for all $j$,
\begin{displaymath}
\Delta_A(T_{ij}) \le 2^{-l-1}p^{-1/2}s_2^{T_i}(A) + 2^{-2l-2}p^{-1}\|A\|_{\{1,2,3\}} \le 2^{-2l}p^{-1}\|A\|_{\{1,2,3\}}
\end{displaymath}
and such that
\begin{displaymath}
F_A^G((x_i,y_i) + T_{ij}) \le F_A^G(T_{ij}) + Ls_2^S(A).
\end{displaymath}

Notice that $N=N_1N_2 \le \exp(L2^{2l}p)$, moreover $T_{ij} - T_{ij}\subset S - S$ and $s_2^{T_{ij}}(A) \le s_2^{T_i}(A) \le 2^{-l}p^{-1}\|A\|_{\{1,2,3\}}$. Since $\#T_{ij} \le \#T_i \le \#S - 1$, to get the covering $S_i$ it is enough to renumerate the sets $T_{ij}$.
\end{proof}

We are now ready to prove Proposition \ref{crucial}.

\begin{proof}[Proof of Proposition \ref{crucial}]
Define the numbers $\Delta_l,\tilde{\Delta}_l$, $l\ge 0$ as
\begin{displaymath}
\Delta_0 = \Delta_A(T),\; \tilde{\Delta}_0 = s_2^T(A)
\end{displaymath}
and
\begin{displaymath}
\Delta_l = 2^{2-2l}p^{-1}\|A\|_{\{1,2,3\}}, \; \tilde{\Delta}_l = 2^{1-l}p^{-1/2}\|A\|_{\{1,2,3\}}.
\end{displaymath}

Assume first that $T \subset \frac{1}{2}[(B_2^n+\sqrt{p}B_1^n)\times (B_2^n+\sqrt{p}B_1^n)]$ and define for $r,l \in \N$,
\begin{align*}
c_T(r,l) = \sup\{&F_A^G(S)\colon S \subset (B_2^n+\sqrt{p}B_1^n)\times (B_2^n+\sqrt{p}B_1^n),\\
& S-S \subset T-T,\#S \le r, \Delta_A(S) \le \Delta_l, s_2^S(A) \le \tilde{\Delta}_l\}.
\end{align*}
We have $c_T(1,l) = 0$. Moreover
\begin{align}
c_T(r,0) \ge \sup\{F_A^G(S)\colon S \subset T, \#S \le r\}.
\end{align}

Notice now, that for any $S$ satisfying the constraints from the definition of $c_T(r,l)$, by Lemma \ref{second_decomposition_lemma}, we
can find a decomposotion $S = \bigcup_{i=1}^N((x_i,y_i) + S_i)$, satisfying (i)--(v). Thus
\begin{align*}
F_A^G(S) &\le \max_{i}F_A^G((x_i,y_i) + S_i) + L\sqrt{\log N}\Delta_A(S) \\
&\le \max_i F_A^G(S_i) + Ls_2^S(A) + L 2^l\sqrt{p}\Delta_l \\
&\le c_T(r-1,l+1) + L\tilde{\Delta}_l + L2^l\sqrt{p}\Delta_l.
\end{align*}
Taking the supremum yields
\begin{displaymath}
c_T(r,l) \le c_T(r-1,l+1) + L\tilde{\Delta}_l + L2^l\sqrt{p}\Delta_l,
\end{displaymath}
which gives
\begin{align*}
c_T(r,0) &\le c_T(1,r-1) + L\sum_{l=0}^\infty (\tilde{\Delta}_l + 2^l\sqrt{p}\Delta_l)\\
&\le L\Big(\sqrt{p}\Delta_A(T) + s_2^T(A) + 2p^{-1/2}\|A\|_{\{1,2,3\}}\Big).
\end{align*}
To finish the proof it is now enough to notice that for $T \subset (B_2^n+\sqrt{p}B_1^n)\times (B_2^n+\sqrt{p}B_1^n)$,
\begin{displaymath}
F_A^G(T) = 4\sup_{r \ge 1} c_{\frac{1}{2}T}(r,0).
\end{displaymath}

\paragraph{Remark} Note that the only place in the above argument where the quantities $\Delta_A(T)$ and $s_2^T(A)$ appear is the first step of the induction, when we pass from $l=0$ to $l=1$. All the other steps contribute just proper multiples of $\|A\|_{\{1,2,3\}}$ which are upper bounds on the parameters $\Delta_A(S)$ and $s_2^S(A)$ of the set $S$ considered there.
\end{proof}

\section{The partition theorem \label{partition_section}}

In this section we present partition results which will allow us to pass from the bounds on expectations
of suprema of Gaussian processes developed so far to empirical processes involving general random variables with bounded fourth moments (in particular all random variables with log-concave tails).
\begin{lem}
\label{cent_decomp}
Let $\alpha$ and $\tilde{\alpha}$ be two norms on $\er^{n^2}$ and $\er^{2n}$ respectively.
For any $p\geq 1$ and $T\subset (B_2^n +\sqrt{p}B_1^n)\times (B_2^n+\sqrt{p}B_1^n)$ we can find a decomposition
$T=\bigcup_{l=1}^N (T_l+(x_l,y_l))$ with $N\leq \exp(Lp)$,
$(x_l,y_l)\in T$ such that for any $(x,y),(\tilde{x},\tilde{y})\in T_l$
\[
\alpha(x\otimes y-\tilde{x}\otimes\tilde{y})\leq
\frac{1}{p}\Ex\alpha(\calE^{1}\otimes \calE^{2})
\]
and
\[
\tilde{\alpha}(x,y)\leq \frac{1}{\sqrt{p}}\Ex\tilde{\alpha}(\calE^{1},\calE^{2}).
\]
\end{lem}
\begin{proof}
Let
\[
M:=\Ex\alpha(\calE^{1}\otimes \calE^{2}) \mbox{ and }
\tilde{M}:= \Ex\tilde{\alpha}(\calE^{1}, \calE^{2}).
\]
Define norm $\beta$ on $\er^{2n}$ by
\[
\beta((x,y))=\Ex\alpha(x\otimes \calE^{2})+\Ex\alpha(\calE^{1}\otimes y).
\]
By Corollary \ref{estentrgen1} with $d=1$, $a = \sqrt{p}$ and $t=p^{-1/2}$ we can decompose $T=\bigcup_{l=1}^{N_0}S_l$ in such a way that
$N_0\leq \exp(Lp)$ and
\[
\beta(x-\tilde{x},y-\tilde{y})\leq \frac{1}{\sqrt{p}}M,\quad
\tilde{\alpha}(x-\tilde{x},y-\tilde{y})\leq \frac{1}{\sqrt{p}}\tilde{M}
\]
for any $(x,y),(\tilde{x},\tilde{y})\in S_l$.
Let us choose any $(x_l,y_l)\in S_l$, put $\tilde{S}_l=S_l-(x_l,y_l)$ and notice that
\begin{align*}
V_2^{\tilde{S}_l}(\alpha,\frac{1}{2\sqrt{p}}) = \frac{1}{4p}M + \frac{1}{2\sqrt{p}}\sup_{(x,y)\in \tilde{S}_l}\beta((x,y)) \le \frac{1}{p}M.
\end{align*}
Hence again by Corollary \ref{estentrgen1} with $t=p^{-1/2}/2$ we can
decompose $\tilde{S}_l=\bigcup_{k=1}^{N_l}T_{l,k}$ with $N_l\leq \exp(Lp)$
and $\alpha(x\otimes y-\tilde{x}\otimes\tilde{y})\leq  \frac{1}{p}M$
for all $(x,y),(\tilde{x},\tilde{y})\in T_{l,k}$.
\end{proof}

\begin{thm}\label{universal_partition}
For any $p\geq 1$ and $T\subset (B_2^n+\sqrt{p}B_1^n) \times (B_2^n+\sqrt{p}B_1^n)$ we can find a decomposition
$T=\bigcup_{l=1}^N (T_l+(x_l,y_l))$ with $N\leq \exp(Lp)$,
$(x_l,y_l)\in T$ such that for any $z_k$,
\[
\Ex \sup_{(x,y)\in T_l} \sum_{ijk}a_{ijk}x_iy_jz_kg_k\leq
\frac{L}{\sqrt{p}}\Big(\sum_{ijk} a_{ijk}^2z_k^{4}\Big)^{1/4}
\Big(\sum_{ijk}a_{ijk}^2\Big)^{1/4}.
\]
\end{thm}

\begin{proof}
Let
\[
\alpha_{z}(x):=(\sum_{k}z_k^2(\sum_{ij}a_{ijk}x_{ij})^{2})^{1/2}
\]
and
\[
\tilde{\alpha}_{z}(x,y):=(\sum_{j,k}z_k^2(\sum_{i}a_{ijk}x_{i})^{2})^{1/2}+
(\sum_{i,k}z_k^2(\sum_{j}a_{ijk}y_{j})^{2})^{1/2}.
\]
Notice that by the Schwarz inequality
\begin{equation}
\label{Schw}
\alpha_{z}(x)\leq \Big(\sum_{ijk} a_{ijk}^2z_k^{4}\Big)^{1/4}\beta(x),\quad
\tilde{\alpha}_{z}(x,y)\leq \Big(\sum_{ijk} a_{ijk}^2z_k^{4}\Big)^{1/4}
\tilde{\beta}(x,y),
\end{equation}
where
\[
\beta(x):=
\Big(\sum_{k}\frac{(\sum_{ij}a_{ijk}x_{ij})^{4}}{\sum_{ij}a_{ijk}^2}\Big)^{1/4}
\]
and
\[
\tilde{\beta}(x,y):=
\Big(\sum_{j,k}\frac{(\sum_{i}a_{ijk}x_{i})^{4}}{\sum_{i}a_{ijk}^2}\Big)^{1/4}
+\Big(\sum_{i,k}\frac{(\sum_{j}a_{ijk}y_{j})^{4}}{\sum_{j}a_{ijk}^2}\Big)^{1/4}.
\]
Notice that (since the 4-th and 2-nd moments of chaoses generated by exponential variables are comparable) we have
\[
\Ex\beta(\calE^{1}\otimes \calE^{2})\leq (\Ex\beta^4(\calE^{1}\otimes \calE^{2}))^{1/4}
\leq L\Big(\sum_{ijk}a_{ijk}^2\Big)^{1/4}
\]
and
\[
\Ex\tilde{\beta}(\calE^{1},\calE^{2})\leq (\Ex\tilde{\beta}^4(\calE^{1},\calE^{2}))^{1/4}
\leq L\Big(\sum_{ijk}a_{ijk}^2\Big)^{1/4}.
\]

Hence by Lemma \ref{cent_decomp} we may decompose
$T=\bigcup_{l=1}^N(\tilde{T}_l+(x_l,y_l))$
with $N\leq \exp(Lp)$,
$(x_l,y_l)\in T$ in such a way that for any $(x,y), (\tilde{x},\tilde{y}) \in \tilde{T}_l$,
\[
\beta(x\otimes y-\tilde{x}\otimes \tilde{y})\leq
\frac{1}{Lp}\Big(\sum_{ijk}a_{ijk}^2\Big)^{1/4}
\mbox{ and }
\tilde{\beta}(x,y)\leq \frac{1}{L\sqrt{p}}\Big(\sum_{ijk}a_{ijk}^2\Big)^{1/4}.
\]
The assertion follows by Proposition \ref{crucial} and (\ref{Schw}).
\end{proof}

\begin{cor} \label{4th_moment} Let $Z_1,\ldots,Z_n$ be independent mean zero random variables. For any $p \ge 1$ there exists a decomposition $(B_2^n + \sqrt{p}B_1^n)^2 = \bigcup_{l \le N} ((x_l,y_l) + T_l)$, where $N \le \exp(Lp)$, $(x_l,y_l) \in (B_2^n + \sqrt{p}B_1^n)^2$
and for every $l$,
\begin{align*}
\E\sup_{(x,y) \in T_l}\sum_{ijk}a_{ijk}x_iy_jZ_k &\le \frac{L}{\sqrt{p}}\Big(\sum_{ijk}a_{ijk}^2\Big)^{1/4}  \E \Big(\sum_{ijk}a_{ijk}^2 Z_k^4\Big)^{1/4}\\ &\le \frac{L}{\sqrt{p}}\|A\|_{\{1,2,3\}}\max_k \|Z_k\|_4.
\end{align*}
\end{cor}

\begin{proof}
It is enough to take the decomposition given by Theorem \ref{universal_partition} and notice that by classical symmetrization inequalities and comparison of Gaussian and Rademacher averages, we have
\begin{align*}
\E\sup_{(x,y) \in T_l}\sum_{ijk}a_{ijk}x_iy_jZ_k &\le 2\E\sup_{(x,y) \in T_l}\sum_{ijk}a_{ijk}x_iy_jZ_k\varepsilon_k \\
&\le \sqrt{2\pi} \E\sup_{(x,y) \in T_l}\sum_{ijk}a_{ijk}x_iy_jZ_kg_k,
\end{align*}
where $\varepsilon_k$ (resp. $g_k$) are sequences of i.i.d Rademacher (resp. standard Gaussian) random variables, independent of the sequence $Z_k$.
\end{proof}

\section{Proof of Theorem \ref{upper} \label{proof_section}}
The case $d= 1$ of the theorem has been proved in \cite{GK}, whereas the case $d=2$ in \cite{L1}, thus it remains to prove the case $d=3$.

To simplify the notation we will write $X_i,Y_j,Z_k$ instead of $X_{i_1}^{1}, X_{i_2}^{2}, X_{i_3}^{3}$ respectively.
Applying the theorem in the (already known) case of chaoses of order two, conditionally on $Z_k$'s yields
\begin{align*}
\E\Big| \sum_{ijk}a_{ijk}X_iY_jZ_k\Big|^p \le& L^p\Big(\E\Big(\Big\|\Big(\sum_k a_{ijk}Z_k\Big)_{ij}\Big\|_{\{1,2\},p}^{\cal N'}\Big)^p \\
&+
\E\Big(\Big\|\Big(\sum_k a_{ijk}Z_k\Big)_{ij}\Big\|_{\{1\}\{2\},p}^{\cal N'}\Big)^p\Big),
\end{align*}
where ${\cal N}' = (N_i^j)_{i\le n, j\le 2}$. Thus by Lemma \ref{tail_est_log_conc} we get
\begin{align}\label{inequality_after_induction}
\Big\| \sum_{ijk}a_{ijk}X_iY_jZ_k\Big\|_p \le& L\Big(\E\Big\|\Big(\sum_k a_{ijk}Z_k\Big)_{ij}\Big\|_{\{1,2\},p}^{\cal N'}
+
\E\Big\|\Big(\sum_k a_{ijk}Z_k\Big)_{ij}\Big\|_{\{1\}\{2\},p}^{\cal N'}\nonumber\\
&+ \|A\|_{\{1,2\}\{3\},p}^{\cal N} + \|A\|_{\{1\}\{2\}\{3\},p}^{\cal N}\Big).
\end{align}

We are therefore left with the problem of estimation of the expectations on the right hand side of the above inequality. This will be achieved in Lemmas \ref{almost_main_lemma} and \ref{main_lemma} below.

Let us first state a simple lemma which will be used repeatedly in the sequel. It is an almost immediate consequence
of the inequality (\ref{no_abs}), therefore we will skip its proof.

\begin{lem} If ${\cal J}$ is a partition of $\{1,2,3\}$ and $\# {\cal J} = r$, then for any $t \ge 1$,
$\|A\|_{{\cal J},tp}^{\cal N} \le t^r\|A\|_{{\cal J},p}^{\cal N}$.
\end{lem}

\begin{lem} \label{almost_main_lemma} Let ${\cal N}' = (N_{i}^j)_{i\le n,j\le 2}$. Then for any $p \ge 2$,
\begin{displaymath}
\E\|(\sum_{k} a_{ijk} Z_k)_{ij}\|_{\{1,2\},p}^{\cal N'} \le L(\|A\|_{\{1,2,3\},p}^{\cal N} + \|A\|_{\{1,2\}\{3\},p}^{\cal N}).
\end{displaymath}
\end{lem}

\begin{proof} By symmetry it is enough to prove that
\begin{align}\label{one_summand}
&\E\sup\Big\{ \sum_{ijk}a_{ijk}Z_kx_{ij}\colon \sum_{i}\hat{N}_i^1((\sum_j x_{ij}^2)^{1/2})\le p\Big\} \\
&\le L(\|A\|_{\{1,2,3\},p}^{\cal N} + \|A\|_{\{1,2\}\{3\},p}^{\cal N}).\nonumber
\end{align}
Moreover, we may and will assume that $\sum_{jk}a_{ijk}^2$ is decreasing in $i$.

Let us first notice that \begin{align}\label{estimation_from_below}
\|A\|_{\{1,2,3\},p}^{\cal N} \ge \frac{1}{L}(\sum_{i\le p}(\sum_{jk} a_{ijk}^2)^{1/2} + \sqrt{p}(\sum_{i> p} \sum_{jk}a_{ijk}^2)^{1/2}).
\end{align}

Let $A_p = \{t\in \R^n \colon \sum_{i} \hat{N}_i^1(t_i) \le p\}$ and note that
\begin{align*}
&\E\sup\{ \sum_{ijk}a_{ijk}Z_kx_{ij}\colon \sum_{i}\hat{N}_i^1((\sum_j x_{ij}^2)^{1/2})\le p\}\\
&= \E\sup\{ \sum_{i}t_i \sqrt{\sum_j (\sum_k a_{ijk}Z_k)^2}\colon t \in A_p\}.
\end{align*}

Define
\begin{align*}
A_p^{1} = \{t\in A_p \colon & |t_i| \le 1\},\\
A_p^{2} = \{t \in A_p \colon & \forall_{l \in \N, l\ge 1} \forall_i\; i \in (2^lp,2^{l+1}p] \Rightarrow (t_i = 0\; \textrm{or}\; |t_i| \ge l^3)\},\\
A_p^{3} = \{t \in A_p \colon & t_i = 0\;\textrm{for $i \le 2p$},\\
&\forall_{l \in \N, l\ge 1} \forall_i\; i \in (2^lp,2^{l+1}p] \Rightarrow (1 \le |t_i| \le l^3\;\textrm{or}\; t_i = 0)\}
\end{align*}
and for $m = 1,2,3$,
\begin{align*}
S_m := \E\sup\{ \sum_{i}t_i \sqrt{\sum_j (\sum_k a_{ijk}Z_k)^2}\colon t \in A_p^{m}\}.
\end{align*}

Since $A_p \subset A_p^{1} + A_p^{2} + A_p^{3}$, we have
\begin{align}\label{split_up}
\E\sup\{ \sum_{ijk}a_{ijk}Z_kx_{ij}\colon \sum_{i}\hat{N}_i^1((\sum_j x_{ij}^2)^{1/2})\le p\} \le S_1 + S_2 + S_3.
\end{align}

\paragraph{Step 1}
For $|t|\le 1$, $\hat{N}_i^1(t) = t^2$, so
\begin{align*}
S_1 &= \E \sup\{\sum_i t_i\sqrt{\sum_j(\sum_k a_{ijk}Z_k)^2}\colon \sum_i t_i^2 \le p, \forall_i \ |t_i|\le 1\}\\
&\le \E \sum_{i\le p} \sqrt{\sum_j(\sum_k a_{ijk}Z_k)^2} + \E \sqrt{p}(\sum_{i > p}\sum_j(\sum_k a_{ijk}Z_k)^2)^{1/2}\\
&\le L\Big(\sum_{i\le p}(\sum_{jk} a_{ijk}^2)^{1/2} + \sqrt{p}(\sum_{i> p} \sum_{jk}a_{ijk}^2)^{1/2}\Big),
\end{align*}
where in the second inequality we used the fact that $\E Z_n^2 \le L$. By (\ref{estimation_from_below}) this implies that
\begin{align}\label{step_1}
S_1 \le L\|A\|_{\{1,2,3\},p}^{\cal N}
\end{align}

\paragraph{Step 2} We will now estimate $S_2$. To this end let us note that since for $t \ge 1$, $\hat{N}_i^1(t) \ge |t|$,
for every $t \in A_p^{2}$, the set $I(t) = {\rm supp}\; t = \{i\le n\colon t_i \neq 0\}$, satisfies
\begin{align*}
\#I(t) \le 3p\;\textrm{and}\; \forall_{l\in \N, l\ge 1}\ \# (I(t) \cap (2^lp,2^{l+1}p]) \le p/l^3.
\end{align*}
Let us denote the family of subsets of $\{1,\ldots,n\}$ satisfying the above conditions by $\mathcal{I}$. We have
\begin{displaymath}
\#\mathcal{I} \le 2^{2p} \prod_{l\ge 1} \Big(\sum_{s\le  p/l^3}\binom{2^lp}{s}\Big) \le 2^{2p}\prod_{l\ge 1}\Big(\frac{e2^lp}{p/l^3}\Big)^{p/l^3} \le L^p.
\end{displaymath}

For each $I \in \mathcal{I}$ let $B_I = {\rm conv} \{t \in \R^n \colon {\rm supp}\; t \subset I, \sum_i \hat{N}_i^1(t_i) \le p\}$. Then
\begin{align*}
S_2 \le \E\max_{I\in \mathcal{I}} \sup_{t \in B_I} \sum_i t_i\sqrt{\sum_j(\sum_k a_{ijk}Z_k)^2}.
\end{align*}

For each $I \in \mathcal{I}$, the set $B_I$ admits a $1/2$-net ${\cal M}_I$ (with respect to the semi-norm induced by $B_I$) of cardinality at most
$5^{\#I} \le 5^{3p}$. By standard approximation arguments we have
\begin{displaymath}
\sup_{t \in B_I} \sum_i t_i\sqrt{\sum_j(\sum_k a_{ijk}Z_k)^2} \le 2\sup_{t \in {\cal M}_I} \sum_i t_i\sqrt{\sum_j(\sum_k a_{ijk}Z_k)^2}.
\end{displaymath}
Therefore
\begin{align*}
S_2 \le \E \sup_{t \in \bigcup_{I\in \mathcal{I}} {\cal M}_I} \sum_i t_i \sqrt{\sum_j(\sum_k a_{ijk} Z_k)^2},
\end{align*}
which by Lemma \ref{union_of_sets} is up to a universal constant majorized by
\begin{align*}
&\sup_{t \in \bigcup_{I\in \mathcal{I}} {\cal M}_I} \sum_i t_i \E \sqrt{\sum_j(\sum_k a_{ijk} Z_k)^2} \\
&+ \sup_{t \in \bigcup_{I\in \mathcal{I}} {\cal M}_I}\sup_{r\colon \sum_k\hat{N}_k^3(r_k) \le Lp} \sum_i t_i  \sqrt{\sum_j(\sum_k a_{ijk} r_k)^2}\\
\le& L\sup_{t\in A_p} t_i\sqrt{\sum_{jk} a_{ijk}^2} + \|A\|_{\{1,2\},\{3\},Lp}^{\cal N} \le L(\|A\|_{\{1,2,3\},p}^{\cal N} + \|A\|_{\{1,2\},\{3\},Lp}^{\cal N}).
\end{align*}
Since for $t \ge 1$, $\|A\|_{\{1,2\},\{3\},tp}^{\cal N} \le t^2\|A\|_{\{1,2\},\{3\},p}^{\cal N}$, the above inequality implies that
\begin{align}\label{step_2}
S_2 \le L(\|A\|_{\{1,2,3\},p}^{\cal N} + \|A\|_{\{1,2\},\{3\},p}^{\cal N}).
\end{align}
\paragraph{Step 3}
For $|t| \ge 1$, $\hat{N}_i^1(t) \ge t$, so
\begin{align*}
S_3 &\le \sum_{l\ge 1} \E\sup\{\sum_{2^lp<i\le 2^{l+1}p} t_i \sqrt{\sum_j (\sum_k a_{ijk}Z_k)^2}\colon \sum_i |t_i| \le p, |t_i| \le l^3 \}\\
& \le L\sum_{l\ge 1} \min(l^3,p) \E\max_{I \subset (2^lp,2^{l+1}p],\#I \le \lceil p/l^3\rceil} \sum_{i\in I} \sqrt{\sum_j (\sum_k a_{ijk}Z_k)^2} \\
&\le L\sum_{l\ge 1} \min(l^3,p) \lceil p/l^3\rceil^{3/4} \E \max_{I \subset (2^lp,2^{l+1}p],\#I \le \lceil p/l^3\rceil}
\Big(\sum_{i\in I} (\sum_j (\sum_k a_{ijk}Z_k)^2)^2\Big)^{1/4}\\
&\le L\sum_{l\ge 1} (pl)^{3/4}
\E \Big(\sum_{2^lp < i \le 2^{l+1}p} (\sum_j  (\sum_k a_{ijk}Z_k)^2)^2\Big)^{1/4}\\
&\le L\sum_{l\ge 1} (pl)^{3/4}
\Big(\sum_{2^lp < i \le 2^{l+1}p} \E (\sum_j  (\sum_k a_{ijk}Z_k)^2)^2\Big)^{1/4}\\
&\le L\sum_{l\ge 1} (pl)^{3/4}
\Big(\sum_{2^lp < i \le 2^{l+1}p} (\sum_{jk}a_{ijk}^2)^2\Big)^{1/4},
\end{align*}
where in the last inequality we used the comparison of the $4$-th and the second moment of norms of linear combinations of independent random variables with log-concave tails.

Now, denote $B = \sqrt{\sum_{i>p} \sum_{ij}a_{ijk}^2}$ and notice that by the assumption on monotonicity of $\sum_{jk}a_{ijk}^2$, we have for $i > p$
\begin{displaymath}
\sum_{jk} a_{ijk}^2 \le \frac{B^2}{i-p}.
\end{displaymath}
Therefore, we have
\begin{align*}
S_3 &\le L\sum_{l\ge 1}(pl)^{3/4} \Big(\sum_{i > 2^lp} \frac{B^4}{(i-p)^2}\Big)^{1/4} \le LB\sum_{l\ge 1}(pl)^{3/4} \frac{1}{(2^lp)^{1/4}}
\le L\sqrt{p}B,
\end{align*}
which by (\ref{estimation_from_below}) implies that
\begin{align}\label{step_3}
S_3 \le L\|A\|_{\{1,2,3\},p}^{\cal N}.
\end{align}
Inequalities (\ref{split_up}-\ref{step_3}) imply (\ref{one_summand}) and conclude the proof of the lemma.
\end{proof}

We will also need the following lemma, proven in \cite{L1} (Corollary 3, therein). We would like to remark in passing that the approach in \cite{L1} was different that in the present article and that the tools developed in the previous sections could be used to give another proof of this lemma
(in the spirit of the argument we provide below for Lemma \ref{main_lemma}). It seems a little bit more natural since Lemmas \ref{almost_main_lemma}
and Lemma \ref{main_lemma} play in the proof of Theorem \ref{upper} for $d=3$ a role analogous to role played by Lemma \ref{suprad1} in the proof of its counterpart for $d=2$.

\begin{lem}[Corollary 3 in \cite{L1}] \label{suprad1} Consider any matrix $A = (a_{ij})_{ij\le n}$ and let ${\cal N}_1 = (N_i^1)_{i\le n}$, ${\cal N}' = (N_{i}^j)_{i\le n,j\le 2}$.
Then, for any $p\geq 2$,
\[
\Ex\Big\|\Big(\sum_{j}a_{ij}Y_j\Big)\Big\|^{{\cal N}_1}_{\{1\},p} \leq L(\|A\|^{{\cal N}'}_{\{1,2\},p} + \|A\|^{{\cal N}'}_{\{1\}\{2\},p}).
\]
\end{lem}

\begin{lem}\label{small_set}
Let ${\cal N}' = (N_{i}^j)_{i\le n,j\le 2}$.
Then
\[
\Ex\Big\|\Big(\sum_{k\le p}a_{ijk}Z_k\Big)_{i,j}\Big\|^{\cal N'}_{\{1\}\{2\},p}\leq
L\|A\|^{\cal N}_{\{1\}\{2\}\{3\},p}.
\]
\end{lem}

\begin{proof}
Consider the norm on $\R^{\lfloor p\rfloor}$ given by
\begin{align*}
\|(z_1,\ldots,z_k)\| = \Big\|\Big(\sum_{k\le p}a_{ijk}z_k\Big)_{i,j}\Big\|^{\cal N'}_{\{1\}\{2\},p}
\end{align*}
and let $K$ be the unit ball of the dual norm $\|\cdot\|_\ast$. Let $M$ be a $1/2$ net in $K$ (with respect to $\|\cdot\|_\ast$) of cardinality not larger than $3^{\lfloor p\rfloor}$ ($M$ exists by standard volumetric arguments). Then for all $z \in \R^{\lfloor p\rfloor}$,
\begin{displaymath}
\|z\| \le 2\sup_{u \in M} \sum_{k\le p} u_k z_k.
\end{displaymath}

Thus
\begin{displaymath}
\Ex\Big\|\Big(\sum_{k\le p}a_{ijk}Z_k\Big)_{i,j}\Big\|^{\cal N'}_{\{1\}\{2\},p}\leq 2 \E \sup_{u \in M} \sum_{k\le p} u_kZ_k,
\end{displaymath}
which by Lemma \ref{union_of_sets} does not exceed
\begin{align*}
&L\sup \{\sum_{k\le p} u_k z_k \colon u \in M-M, \sum_{k\le p} \hat{N}_k^{3}(z_k) \le p\} \le L\sup\{\|z\|\colon \sum_{k\le p} \hat{N}_k^{3}(z_k) \le p\}\\
& = L\|A\|^{\cal N}_{\{1\}\{2\}\{3\}}.
\end{align*}
\end{proof}

\begin{lem} \label{main_lemma} Let ${\cal N}' = (N_{i}^j)_{i\le n,j\le 2}$. Then for any $p \ge 2$,
\begin{align}\label{main_estimate}
&\E\|(\sum_{k} a_{ijk} Z_k)_{ij}\|_{\{1\}\{2\},p}^{\cal N'} \nonumber\\
&\le L \Big(\|A\|_{\{1,2,3\},p}^{\cal N}+ \|A\|_{\{1\}\{2,3\},p}^{\cal N} +
\|A\|_{\{2\}\{1,3\},p}^{\cal N} + \|A\|_{\{1\}\{2\}\{3\},p}^{\cal N}\Big).
\end{align}
\end{lem}

\begin{proof}
Let us first notice that it's enough to prove the formally weaker estimate
\begin{align}\label{reduced_estimate}
&\E\|(\sum_{k} a_{ijk} Z_k)_{ij}\|_{\{1\},\{2\},p}^{\cal N'} \nonumber\\
&\le L \Big(\sqrt{p}\|A\|_{\{1,2,3\}} + \|A\|_{\{1\}\{2,3\},p}^{\cal N} +
\|A\|_{\{2\}\{1,3\},p}^{\cal N} + \|A\|_{\{1\}\{2\}\{3\},p}^{\cal N}\Big).
\end{align}

Indeed, suppose that the above inequality holds for all triple-indexed matrices, and assume additionaly (without loss of generality) that $\sum_{ij} a_{ijk}^2$ decreases in $k$.
We have
\begin{align*}
\E\|(\sum_{k} a_{ijk} Z_k)_{ij}\|_{\{1\},\{2\},p}^{\cal N'} \le \E\|(\sum_{k\le p} a_{ijk} Z_k)_{ij}\|_{\{1\},\{2\},p}^{\cal N'} + \E\|(\sum_{k>p} a_{ijk} Z_k)_{ij}\|_{\{1\},\{2\},p}^{\cal N'}.
\end{align*}
By Lemma \ref{small_set} we have
\begin{displaymath}
\E\|(\sum_{k\le p} a_{ijk} Z_k)_{ij}\|_{\{1\}\{2\},p}^{\cal N'} \le L\|A\|^{\cal N}_{\{1\}\{2\}\{3\},p}.
\end{displaymath}
Moreover, by our assumption
\begin{displaymath}
\E\|(\sum_{k> p} a_{ijk} Z_k)_{ij}\|_{\{1\}\{2\},p}^{\cal N'} \le L\sqrt{p}\Big(\sum_{k>p}\sum_{ij} a_{ijk}^2\Big)^{1/2}+ \sum_{{\cal J}\in P_d, {\cal J}\neq \{\{1,2,3\}\} \atop
{\cal J}\neq \{\{1,2\},\{3\}\}} \|A\|_{{\cal J},p}^{\cal N}.
\end{displaymath}
Monotonicity of $\sum_{ij}a_{ijk}^2$ implies that
\begin{displaymath}
\sqrt{p}\Big(\sum_{k>p}\sum_{ij} a_{ijk}^2\Big)^{1/2} \le \sup\{\sum_k t_k\sqrt{\sum_{ij}a_{ijk}^2}\colon \sum_k t_k^2 \le p, |t_k|\le 1\} \le \|A\|_{\{1,2,3\},p}^{\cal N},
\end{displaymath}
which together with the previous three inequalities proves (\ref{main_estimate}).

We will now prove (\ref{reduced_estimate}). To this end let us denote
$$A_p^j = \{t\in \R^n\colon \sum_{i=1}^n \hat{N}_i^j(t_i) \le p\},\; j=1,2,3.$$
Since $\hat{N}_i^j(t) \ge |t|$ for $t> 1$, it is easy to see that $A_p^j \subset \sqrt{p}B_2^n + pB_1^n$. Hence, by Corollary \ref{4th_moment} and the fact that $\E Z_k^4 \le L$, there exists a partition

\begin{displaymath}
A_p^1\times A_p^2 = \bigcup_{l\le N} ((x^l,y^l)+T_l),
\end{displaymath}
with $N \le \exp(Lp)$, $(x^l,y^l) \in A_p\times A_p$, such that
\begin{align}\label{decomposition}
\max_{l\le N}\E\sup_{(x,y) \in T_l} \sum_{ijk}a_{ijk}x_iy_jZ_k \le L\sqrt{p}\|A\|_{\{1,2,3\}}.
\end{align}

Now, by Lemma \ref{union_of_sets},
\begin{align*}
&\E\sup_{(x,y) \in A_p^1\times A_p^2} \sum_{ijk}a_{ijk}x_iy_jZ_k \\
&\le \max_{l \le N}\E\sup_{(x,y) \in (x^l,y^l) + T_l} \sum_{ijk}a_{ijk}x_iy_jZ_k + 2\sup_{(x,y)\in A_p^1\times A_p^2, z\in A_{Lp}^3} \sum_{ijk} a_{ijk} x_iy_jz_k \\
&\le \max_{l \le N}\E\sup_{(x,y) \in (x^l,y^l) + T_l} \sum_{ijk}a_{ijk}x_iy_jZ_k+ L\|A\|_{\{1\}\{2\}\{3\},p}^{\cal N},
\end{align*}
where in the second inequality we used the fact that $A_{Lp}^3 \subset LA_p^3$.

Thus it remains to estimate $\max_{l \le N}\E\sup_{(x,y) \in (x^l,y^l) + T_l}\sum_{ijk}a_{ijk}x_iy_jZ_k$.
Denote by $\pi_1(T),\pi_2(T)$ respectively projections of $T_l$ onto the first $n$ and the last $n$ coordinates and let ${\cal N}_j = (N_i^j)_{i\le n}$, $j=1,2$. We have
\begin{align*}
&\E\sup_{(x,y) \in (x^l,y^l) + T_l}\sum_{ijk}a_{ijk}x_iy_jZ_k \\
&\le \E\sup_{(x,y) \in T_l}\sum_{ijk}a_{ijk}x_iy_jZ_k + \E \sup_{x\in \pi_1(T)}\sum_{ijk}a_{ijk}x_iy^l_jZ_k + \E\sup_{y\in \pi_2(T)} \sum_{ijk}a_{ijk}x_i^ly_jZ_k\\
&\le L\sqrt{p}\|A\|_{\{1,2,3\}} + 2\|(\sum_{ik}a_{ijk}x_i^l Z_k)_j\|_{\{2\},p}^{{\cal N}_2} + 2\|(\sum_{jk}a_{ijk}y^l_jZ_k)_i\|_{\{1\},p}^{{\cal N}_1}\\
&\le L\sqrt{p}\|A\|_{\{1,2,3\}} + L\|A\|_{\{1\}\{2,3\},p}^{\cal N} + L\|A\|_{\{2\}\{1,3\},p}^{\cal N} + L\|A\|_{\{1\}\{2\}\{3\},p}^{{\cal N}},
\end{align*}
where the second inequality follows from (\ref{decomposition}) and the fact that $\pi_j(T_l) \subset LA^j_p$, and the third inequality from Lemma
\ref{suprad1} (applied to $Z_k$ instead of $Y_j$, which corresponds to an appropriate permutation of the array $N_i^j$) .
This proves (\ref{reduced_estimate}) and ends the proof of the lemma.
\end{proof}

\begin{proof}[Conclusion of the proof of Theorem \ref{upper}]
By lemmas \ref{almost_main_lemma} and \ref{main_lemma}, the right hand side of (\ref{inequality_after_induction})
does not exceed
\begin{displaymath}
\sum_{{\cal J}\in P_3}\|(a_{ijk})\|_{{\cal J},p}^{{\cal N}},
\end{displaymath}
which ends the proof.
\end{proof}

\section{Proof of Theorem \ref{exponential} \label{proof_section_exp}}
In this section we restrict our attention to the special case of symmetric exponential variables and consider polynomial chaoses of arbitrary order.
For exponential variables, the function $N_i^j(t) = t$, which allows us to replace quantities $\|(a_\ii)\|_{{\cal J},p}^{\cal N}$ by simpler quantities.

\begin{prop} If for all $i\le n, j \le d$, $N_{i}^j(t) = t$, then for every $\mathcal{J} = \{J_1,\ldots,J_k\} \in P_d$ and every $p \ge 2$,
\begin{align*}
&L_d^{-1}\sum_{I \in Q(\mathcal{J})} p^{\#I^c + (k - \#I^c)/2}\max_{\ii_{I^c}}\|(a_\bfi)_{\ii_I}\|_{S({\cal J},I)}\\
&\le \|(a_\bfi)\|_{\mathcal{J},p}^{\cal N}\le
L_d\sum_{I \in Q(\mathcal{J})} p^{\#I^c + (k - \#I^c)/2}\max_{\ii_{I^c}}\|(a_\bfi)_{\ii_I}\|_{S({\cal J},I)}
\end{align*}
where $Q(\mathcal{J}) = \{I \subset \{1,\ldots,d\}\colon \forall_{i \le k} \#(I^c\cap J_i) \le 1\}$ and $S({\cal J},I)$ is the partition of $I$ obtained from $\cal{J}$ by removing from the sets $J_i$ all the elements of $I^c$.
\end{prop}
\begin{proof} It is enough to prove that
\begin{align*}
&L^{-1}\Big(p\max_{i_1}\|(a_{\ii})_{i_2,\ldots,i_d}\|_{\{2,\ldots,d\}} + \sqrt{p} \|(a_\ii)\|_{\{1,\ldots,d\}}\Big) \\
&\le \sup\{\sum_\ii a_\ii x_\ii \colon \sum_{i_1}\min((\sum_{i_2,\ldots,i_d} x_\ii^2)^{1/2},\sum_{i_2,\ldots,i_d} x_\ii^2) \le p \}\\
&\le L\Big(p\max_{i_1}\|(a_{\ii})_{i_2,\ldots,i_d}\|_{\{2,\ldots,d\}} + \sqrt{p} \|(a_\ii)\|_{\{1,\ldots,d\}}\Big).
\end{align*}
The proposition follows easily by an iterative application of this inequality.

To prove the above inequality it suffices to notice that
\begin{align*}
&\{x_\ii \colon \sum_{i_1}\hat{N}_{i_1}^1((\sum_{i_2,\ldots,i_d} x_\ii^2)^{1/2})\le p\}\\
 =& \{z_{i_1}y_\ii\colon \sum_{i}\min(|z_i|,z_i^2) \le p, \forall_{i_1} \sum_{i_2,\ldots,i_d} y_{\ii}^2 \le 1\}
\end{align*}
and
\begin{displaymath}
(\sqrt{p}B_2^n)\cup (pB_1^n)\subset \{z \in \R^n \colon \sum_{i}\min(|z_i|,z_i^2)\le p\} \subset \sqrt{p}B_2^n + pB_1^n.
\end{displaymath}
We will leave the details to the reader.
\end{proof}
For a nonempty set $I$, let us denote by $P_I$ the set of all partitions of $I$ into pairwise disjoint, nonempty sets. In particular $P_{\{1,\ldots,d\}} = P_d$, $P_\emptyset = \{\emptyset\}$.

The above proposition yields the following
\begin{cor}
If for all $i\le n, j \le d$, $N_{i}^j(t) = t$, then for every $p \ge 2$,
\begin{align*}
&L_d^{-1} \sum_{I \subset \{1,\ldots,d\}} \sum_{{\cal J} \in P_I}p^{\#I^c + \#{\cal J}/2}\max_{\bfi_{I^c}}\|(a_\bfi)_{\ii_{I}}\|_{\cal J}\\
 &\le \sum_{{\cal J}\in P_d}\|(a_{\bfi})\|_{{\cal J},p}^{\cal N} \le L_d \sum_{I \subset \{1,\ldots,d\}} \sum_{{\cal J} \in P_I}p^{\#I^c + \#{\cal J}/2}\max_{\bfi_{I^c}}\|(a_\bfi)_{\ii_{I}}\|_{\cal J}.
\end{align*}
\end{cor}

From the above corollary and Theorem \ref{lower} it follows that to prove Theorem \ref{exponential} it is enough to demonstrate the following
\begin{prop} \label{exponential_1} If $(X_i^j)_{i\le n, j \le d}$ are independent symmetric exponential random variables, then for every $p \ge 2$,
\begin{align}\label{exponential_upper}
\Big\|\sum_{\bfi}a_{\bfi}X_{i_1}^{1}\cdots X_{i_d}^{d}\Big\|_p \le L_d \sum_{I \subset \{1,\ldots,d\}} \sum_{{\cal J} \in P_I}p^{\#I^c + \#{\cal J}/2}\max_{\bfi_{I^c}}\|(a_\bfi)_{\ii_{I}}\|_{\cal J}.
\end{align}
\end{prop}

The proof of Proposition \ref{exponential_1} will be based on induction with respect to $d$. It will require several additional lemmas. Throughout the rest of this section we will assume that $(X_i^j)_{i\le n, j \le d}$ are independent symmetric exponential random variables.

\begin{lem}\label{main_lemma_exp} For any $d= 2,3,\ldots$,
\begin{align*}
&\E\|(\sum_{i_d} a_\bfi X_{i_d}^d)_{\ii_{\{d\}^c}}\|_{\{1\}\ldots\{d-1\}} \\
&\le L_d\sum_{{\cal J} \in P_d}p^{(1+\#{\cal J} - d)/2}\|(a_\bfi)\|_{\cal J} + L_d\sum_{{\cal J} \in P_{d-1}}p^{1+(1+\#{\cal J} - d)/2}\max_{i_d}\|(a_\bfi)_{\bfi_{\{d\}^c}}\|_{\cal J}.\nonumber
\end{align*}
\end{lem}

We will need the following technical fact.

\begin{lem}\label{comparison_lemma}
Let $Y_i^{(1)}$ be independent standard symmetric exponential variables and $Y_i^{(2)} = g_i^2$, $Y_{i}^{(3)} = g_i \tilde{g}_i$, where $g_i,\tilde{g}_i$ are i.i.d $\mathcal{N}(0,1)$ variables and $\varepsilon_i$ - i.i.d Rademacher variables independent of $Y_i^{(j)}$. Then for any normed space $E$ and any vectors $v_1,\ldots,v_n\in E$
the quantities
\begin{displaymath}
\Ex\big\|\sum_i v_i \varepsilon_i Y_i^{(j)}\Big\|,\; j=1,2,3,
\end{displaymath}
are comparable up to universal multiplicative factors.
\end{lem}

\begin{proof}
Since we can symmetrize all variables, and by the contraction principle and Jensen's inequality
\begin{displaymath}
\Ex\Big\|\sum_i v_i \varepsilon_i |Y_i^{(j)}|\Big\| \ge c\Ex\Big\|\sum_i v_i\varepsilon_i\Big\|,
\end{displaymath}
it is enough to show that one can define copies of the variables $Y_i^{(j)}$ (which we will identify with the variables) on a common probability space in such a way that for any
$j,k = 1,2,3$,
\begin{displaymath}
|Y_i^{(j)}| \le L(1 + |Y_i^{(k)}|).
\end{displaymath}
This is possible by using the inverse of the distribution function, since
\begin{displaymath}
\Pr(|Y_i^{(1)}| \ge t) = e^{-t},
\end{displaymath}
\begin{displaymath}
L^{-1}e^{-Lt} \le \Pr(Y_i^{(2)} \ge t) \le e^{-t/2}
\end{displaymath}
and
\begin{displaymath}
L^{-1}e^{-Lt} \le \Pr(Y_i^{(3)} \ge t) \le 2e^{-t/2}
\end{displaymath}
\end{proof}

The proof of Lemma \ref{main_lemma_exp} will be based on a conditional application of the following result from \cite{L2} (see \cite{A} for a similar approach in the context of moment inequalities for $U$-statistics).

\begin{lem}[\cite{L2}, Theorem 2]\label{Gaussian_averages_1}
For any $p \ge 2$,
\begin{eqnarray*}
\E\|(\sum_{i_d}
a_{\ii}g_{i_d})_{\ii_{I_{d-1}}}\|_{\{1\}\dots\{d-1\}} &\le& L_d
\sum_{\mathcal{J} \in \mathcal{P}_{I_d}} p^{(1 + \#\mathcal{J}
- d)/2}\|(a_\ii)\|_\mathcal{J}.
\end{eqnarray*}
\end{lem}

\begin{proof}[Proof of Lemma \ref{main_lemma_exp}]
Lemmas \ref{comparison_lemma} and \ref{Gaussian_averages_1} give
\begin{align}\label{exponential_auxiliary_0}
\E\|(\sum_{i_d} a_\bfi X_{i_d}^d)\|_{\{1\}\ldots\{d-1\}} &\le L\E\|(\sum_{i_d} a_\bfi g_{i_d}\tilde{g}_{i_d})\|_{\{1\}\ldots\{d-1\}} \nonumber\\
&\le L_d\sum_{{\cal J}\in P_{d}}p^{(1+\#{\cal J}-d)/2}\E \|(a_\bfi g_{i_d})\|_{\cal J}.
\end{align}
Take ${\cal J} \in P_d$ of the form ${\cal J} = \{I_1\cup\{d\},\ldots, I_k\}$ where $\{I_1,\ldots,I_k\}\setminus\{\emptyset\} \in P_{d-1}$.
We have
\begin{align*}
\E\|(a_\bfi & g_{i_d})\|_{\cal J}^2 = \E \sup_{\|x_{\ii_{I_j}}^{j}\|_2\le 1, j=2,\ldots,k}\sum_{i_d}g_{i_d}^2\sum_{\bfi_{I_1}}\Big(\sum_{\ii_{(I_1\cup\{d\})^c}} a_\bfi\prod_{j=2}^k x_{\ii_{I_j}}^j\Big)^2 \\
&\le \|(a_\bfi)\|_{\cal J}^2 + \E \sup_{\|x_{\ii_{I_j}}^{j}\|_2\le 1, j=2,\ldots,k}\sum_{i_d}(g_{i_d}^2-1)\sum_{\bfi_{I_1}}\Big(\sum_{\ii_{(I_1\cup\{d\})^c}} a_\bfi\prod_{j=2}^k x_{\ii_{I_j}}^j\Big)^2.
\end{align*}
Since $\E g_{i_d}^2 = 1$, standard symmetrization arguments applied to the second term on the right hand side give
\begin{align}\label{exponential_auxiliary_1}
\E\|(a_\bfi g_{i_d})\|_{\cal J}^2 &\le \|(a_\bfi)\|_{\cal J}^2 + 2\E \sup_{\|x_{\ii_{I_j}}^{j}\|_2\le 1, j=2,\ldots,k}\sum_{i_d}\varepsilon_{i_d}g_{i_d}^2\sum_{\bfi_{I_1}}\Big(\sum_{\ii_{(I_1\cup\{d\})^c}} a_\bfi\prod_{j=2}^k x_{\ii_{I_j}}^j\Big)^2\\
&\le \|(a_\bfi)\|_{\cal J}^2 + L\E \sup_{\|x_{\ii_{I_j}}^{j}\|_2\le 1, j=2,\ldots,k}\sum_{i_d}g_{i_d}\tilde{g}_{i_d}\sum_{\bfi_{I_1}}\Big(\sum_{\ii_{(I_1\cup\{d\})^c}} a_\bfi\prod_{j=2}^k x_{\ii_{I_j}}^j\Big)^2,\nonumber
\end{align}
where in the second inequality we used again Lemma \ref{comparison_lemma}.
Let now
\begin{align*}
M = \max_{i_d} \sup_{\|x_{\ii_{I_j}}^{j}\|_2\le 1, j=2,\ldots,k} \sqrt{\sum_{\bfi_{I_1}} (\sum_{\ii_{(I_1\cup\{d\})^c}} a_\bfi\prod_{j=2}^k x_{\ii_{I_j}}^j)^2} = \max_{i_d} \|(a_\bfi)_{\bfi_{\{d\}^c}}\|_{I_1,\ldots,I_k}
\end{align*}
and for fixed $g_{i_d}$ consider functions $\varphi_k \colon \R \to \R$, given by the formula
\begin{displaymath}
\varphi_{i_d}(t) = \begin{cases}\frac{t^2}{2M g_{i_d}}\quad \textrm{for}\quad |t| \le |g_{i_d}| M,\\
g_{i_d}M/2 \quad \textrm{for}\quad |t| > |g_{i_d}| M.
\end{cases}
\end{displaymath}
We have $|\varphi_{i_d}'(t)| = |t|/(|g_{i_d}|M) \le 1$ for $|t| \le |g_{i_d}|M$, moreover $\varphi_{i_d}$ is constant for  $t\ge |g_{i_d}|M$, so $\varphi_{i_d}$ is $1$-Lipschitz. Thus, by the contraction principle (see Corollary 3.17 in \cite{LT}),
\begin{align}\label{contraction}
&\E_{\tilde{g}} \sup_{\|x_{\ii_{I_j}}^{j}\|_2\le 1, j=2,\ldots,k}\Big|\sum_{i_d}\tilde{g}_{i_d}\varphi_{i_d}\Big(g_{i_d}\Big(\sum_{\bfi_{I_1}} \Big(\sum_{\ii_{(I_1\cup\{d\})^c}} a_\bfi\prod_{j=2}^k x_{\ii_{I_j}}^j\Big)^2\Big)^{1/2}\Big)\Big|\\
&\le
4\E_{\tilde{g}}\sup_{\|x_{\ii_{I_j}}^{j}\|_2\le 1, j=2,\ldots,k}\Big|\sum_{i_d}\tilde{g}_{i_d} g_{i_d}\Big(\sum_{\bfi_{I_1}} \Big(\sum_{\ii_{(I_1\cup\{d\})^c}} a_\bfi\prod_{j=2}^k x_{\ii_{I_j}}^j\Big)^2\Big)^{1/2}\Big|,\nonumber
\end{align}
which implies that
\begin{align}\label{exponential_auxiliary_2}
&\E_{\tilde{g}} \sup_{\|x_{\ii_{I_j}}^{j}\|_2\le 1, j=2,\ldots,k}\Big|\sum_{i_d}\tilde{g}_{i_d}g_{i_d}\sum_{\bfi_{I_1}}\Big(\sum_{\ii_{(I_1\cup\{d\})^c}} a_\bfi\prod_{j=2}^k x_{\ii_{I_j}}^j\Big)^2\Big|\\
&\le 8M\E_{\tilde{g}} \sup_{\|x_{\ii_{I_j}}^{j}\|_2\le 1, j=2,\ldots,k}\Big|\sum_{i_d}\tilde{g}_{i_d} g_{i_d}\Big(\sum_{\bfi_{I_1}} \Big(\sum_{\ii_{(I_1\cup\{d\})^c}} a_\bfi\prod_{j=2}^k x_{\ii_{I_j}}^j\Big)^2\Big)^{1/2}\Big|.\nonumber
\end{align}

Denote $T = \prod_{j=2}^k B_{H_j}$, where $B_{H_j}$ is the unit ball of the Hilbert space $\bigotimes_{l \in I_j} \R^n$. For $t \in T$, $t = (x_{\bfi_{I_j}}^j)_{j=2}^k$ let
\begin{displaymath}
X_t = \sum_{i_d}\tilde{g}_{i_d} g_{i_d}\Big(\sum_{\bfi_{I_1}} (\sum_{\ii_{(I_1\cup\{d\})^c}} a_\bfi\prod_{j=2}^k x_{\ii_{I_j}}^j)^2\Big)^{1/2}.
\end{displaymath}
Then, conditionally on $g_{i_d}$, $(X_t)_{t\in T}$ is a Gaussian process. It induces a metric on $T$ given by
\begin{displaymath}
d_X(t,s) = \|X_t - X_s\|_2.
\end{displaymath}
More explicitly if $t = (x_{\bfi_{I_j}}^j)_{j=2}^k$, $s = (y_{\bfi_{I_j}}^j)_{j=2}^k$, then
\begin{align*}
&d_X(t,s)^2 \\
&= \sum_{i_d} g_{i_d}^2 \Big[\Big(\sum_{\bfi_{I_1}}\Big(\sum_{\ii_{(I_1\cup\{d\})^c}} a_\bfi\prod_{j=2}^k x_{\ii_{I_j}}^j\Big)^2\Big)^{1/2} - \Big(\sum_{\bfi_{I_1}}\Big(\sum_{\ii_{(I_1\cup\{d\})^c}} a_\bfi\prod_{j=2}^k y_{\ii_{I_j}}^j\Big)^2\Big)^{1/2}\Big]^2\\
&\le \sum_{i_d} g_{i_d}^2 \sum_{\bfi_{I_1}}\Big(\sum_{\ii_{(I_1\cup\{d\})^c}} a_\bfi\Big(\prod_{j=2}^k x_{\ii_{I_j}}^j - \prod_{j=2}^k y_{\ii_{I_j}}^j\Big)\Big)^2,
\end{align*}
where to obtain the last inequality for each fixed $i_d$ we used the triangle inequality in the space $\ell_2(\{1,\ldots,n\}^{I_1})$ for vectors $a_{\ii_{I_1}} = \sum_{\ii_{(I_1\cup\{d\})^c}} a_\bfi\prod_{j=2}^k x_{\ii_{I_j}}^j$ and $b_{\ii_{I_1}} = \sum_{\ii_{(I_1\cup\{d\})^c}} a_\bfi\prod_{j=2}^k y_{\ii_{I_j}}^j$.

Now, the right-hand side above is equal to $d_{\tilde{X}}(t,s) = \|\tilde{X}_t - \tilde{X}_s\|_2$, where $(\tilde{X}_t)_{t\in T}$ is a (conditionally) Gaussian process defined as
\begin{displaymath}
\tilde{X}_t = \sum_{\ii_{I_1\cup \{d\}}} g_{i_d} \tilde{g}_{\ii_{I_1\cup\{d\}}} \sum_{\ii_{(I_1\cup\{d\})^c}} a_\bfi\prod_{j=2}^k x_{\ii_{I_j}}^j
\end{displaymath}
where $t=(x_{\bfi_{I_j}}^j)_{j=2}^k$ and $(\tilde{g}_{\ii_{I_1\cup\{d\}}})_{\ii_{{I_1\cup\{d\}}}}$ is an array of i.i.d. standard Gaussian variables independent of $g_{i_d}$.

Thus by the Slepian lemma we have
\begin{displaymath}
\E_{\tilde{g}}\sup_{t \in T} X_{t} \le \E_{\tilde{g}}\sup_{t \in T} \tilde{X}_t.
\end{displaymath}
Moreover, since $0 \in T$, $X_0 = 0$ and $T$ is symmetric with respect to the origin, we have
\begin{displaymath}
\E\sup_{t\in T}|X_t| = \E\max (\sup_{t\in T} X_t,\sup_{t\in T}(-X_t)) \le \E \sup_{t\in T} X_t + \E\sup_{t \in T} (-X_t) = 2\sup_{t \in T} X_t.
\end{displaymath}
Thus, we have
\begin{align*}
&\E_{\tilde{g}} \sup_{\|x_{\ii_{I_j}}^{j}\|_2\le 1, j=2,\ldots,k}\Big|\sum_{i_d}\tilde{g}_{i_d} g_{i_d}\Big(\sum_{\bfi_{I_1}} \Big(\sum_{\ii_{(I_1\cup\{d\})^c}} a_\bfi\prod_{j=2}^k x_{\ii_{I_j}}^j\Big)^2\Big)^{1/2}\Big| \\
&\le  2\E_{\tilde{g}} \sup_{\|x_{\ii_{I_j}}^{j}\|_2\le 1, j=2,\ldots,k}\sum_{\ii_{I_1\cup \{d\}}}  \tilde{g}_{\ii_{I_1\cup\{d\}}} \sum_{\ii_{(I_1\cup\{d\})^c}} g_{i_d}a_\bfi\prod_{j=2}^k x_{\ii_{I_j}}^j\\
&\le L_d\sum_{{\cal K} \in P_d} p^{(1+\#{\cal K} - k)/2} \|(a_\bfi g_{i_d})\|_{\cal K},
\end{align*}
where the last inequality follows from another application of Lemma \ref{Gaussian_averages_1}, conditionally on $g_{i_d}$.
Going now back to (\ref{exponential_auxiliary_1}) and (\ref{exponential_auxiliary_2}), we obtain that for all $\varepsilon \in (0,1)$ and all $\mathcal{J} = \{I_1\cup\{d\},I_2,\ldots,I_k\} \in P_{d}$,
\begin{align*}
&p^{(1+k-d)/2} \E\|(a_\bfi g_{i_d})\|_{\cal J} \\
\le & p^{(1+k-d)/2}\|(a_\bfi)\|_{\cal J} \\
&+ L_d \sqrt{p^{(2+k-d )/2}\max_{i_d} \|(a_\bfi)_{\bfi_{\{d\}^c}}\|_{\cal J'}}\sqrt{\sum_{{\cal K} \in P_d}p^{(1+\#{\cal K} -d)/2}\E\|(a_{\bfi}g_{i_d})\|_{\cal K}}\\
\le& p^{(1+k-d)/2}\|(a_\bfi)\|_{\cal J} \\
&+L_d p^{1+ (1+\#{\cal J}'-d)/2}\varepsilon^{-1}\max_{i_d} \|(a_\bfi)_{\bfi_{\{d\}^c}}\|_{\cal J'} + \varepsilon L_d  \sum_{{\cal K} \in P_d}p^{(1+\#{\cal K} - d)/2}\E\|(a_{\bfi}g_{i_d})\|_{\cal K},
\end{align*}
where ${\cal J}' = \{I_1,\ldots,I_k\}\setminus\{\emptyset\}$.
Summing the above inequalities over all ${\cal J} \in P_d$ and choosing $\varepsilon$ to be a sufficiently small number depending on $d$, we get
\begin{align*}
\sum_{{\cal J} \in P_d} p^{(1+\#{\cal J} - d)/2}\E\|(a_\bfi g_{i_d})\|_{\cal J} &\le L_d\sum_{{\cal J} \in P_d} p^{(1+\#{\cal J} - d)/2}\|(a_\bfi)\|_{\cal J} \\
&+ L_d\sum_{{\cal J} \in P_{d-1}} p^{1+(1+{\cal J} - d)/2}\max_{i_d}\|(a_{\bfi})_{\ii_{\{d\}^c}}\|_{\cal J}.
\end{align*}
Together with (\ref{exponential_auxiliary_0}) this ends the proof of the lemma.
\end{proof}

To prove Proposition \ref{exponential_1} we will also use a technical fact proved in \cite{A} in greater generality (see Lemma 5 therein).
\begin{lem}\label{sumy_na_maxima}
For $\alpha > 0$ and arbitrary nonnegative numbers $r_{i_1,\ldots,i_d}$ and $p > 1$ we have
\begin{displaymath}
p^{\alpha p}\sum_{\ii} r_\ii^p \le L_d^p p^{\alpha
d}\left[p^{\alpha p}\max_{\ii} r_\ii^p + \sum_{I \subsetneq
\{1,\ldots,d\}} p^{\#I p} \max_{\ii_I}(\sum_{\ii_{I^c}}
 r_\ii)^p \right].
\end{displaymath}
\end{lem}
\begin{proof}[Proof of Proposition \ref{exponential_1}] The argument is similar to the proof of Theorem 6 in \cite{A} therefore we will only sketch the main steps.

Since for $p = 2$ the proposition is trivial (recall that $(\|(a_\ii)\|_{\{1,\ldots,d\}} = (\sum_{\ii} a_\ii^2)^{1/2}$), we will assume that $p > 2$.

Let us first note that to prove the proposition it is enough to show that
\begin{align}\label{no_max}
\E|\sum_{\bfi}a_{\bfi}X_{i_1}^{1}\cdots X_{i_d}^{d}\Big|^p \le L_d^p \sum_{I \subset \{1,\ldots,d\}} \sum_{{\cal J} \in P_I}p^{p(\#I^c + \#{\cal J}/2)}\sum_{\bfi_{I^c}}\|(a_\bfi)_{\ii_{I}}\|_{\cal J}^p.
\end{align}

Indeed, for fixed $I$ let us apply Lemma \ref{sumy_na_maxima} (with $p/2$ instead of $p$, $\#I^c$ instead of $d$ and $r_{\ii_{I^c}} = \|(a_\ii)_{\ii_I}\|_{\cal J}^2$). We get
\begin{align*}
&\sum_{\bfi_{I^c}}\|(a_\bfi)_{\ii_{I}}\|_{\cal J}^p \\
&\le L_{\# I^c}^p(p/2)^{\alpha \#I^c}\Big(\max_{\bfi_{I^c}}\|(a_\bfi)_{\ii_{I}}\|_{\cal J}^p
+\sum_{J\subsetneq I^c}(p/2)^{\#Jp/2 - \alpha p/2}\max_{\ii_{J}}(\sum_{\bfi_{I^c\setminus J}}\|(a_{\ii})_{\ii_{I}}\|_{\mathcal{J}}^2)^{p/2}\Big).
\end{align*}

Note that we have $\sum_{\ii_{I^c\setminus J}} \|(a_\ii)_{\ii_I}\|_{\cal J}^2 \le \sum_{\ii_{J^c}} a_\ii^2 = \|(a_\ii)_{\ii_{J^c}}\|_{\{J^c\}}^2$ and
that $p^{\#J + 1/2} \max_{\ii_J} \|(a_\ii)_{\ii_{J^c}}\|_{\{J^c\}}$ appears among the summands on the right hand side of (\ref{exponential_upper}). Thus the above inequality with $\alpha$ sufficiently large (depending only on $d$) implies that the right hand side of (\ref{no_max}) is majorized by the $p$-th power of the right hand side of the inequality asserted in the proposition (we use the fact that if $\alpha$ depends only on $d$ then $p^{\alpha\#I^c} \le L_d^p$).

It remains to prove (\ref{no_max}). We will proceed by induction on $d$. For $d=1$, the proposition (which is stronger than (\ref{no_max}) for $d=1$) is a special case of Theorem \ref{upper} (it also follows from the Gluskin-Kwapie{\'n} estimate).

Let us thus assume that (\ref{no_max})  holds for chaoses of order at most $d-1$. We will show that then it holds for chaoses of order $d$.
Applying the induction assumption conditionally on $(X_i^{d})_i$ together with the Fubini theorem and Lemma \ref{union_of_sets} we obtain
\begin{align*}
&\E|\sum_{\bfi}a_{\bfi}X_{i_1}^{1}\cdots X_{i_d}^{d}\Big|^p \\
\le& L_{d-1}^p \sum_{I \subset \{1,\ldots,d-1\}} \sum_{{\cal J} \in P_I}p^{p(d-1 - \#I + \#{\cal J}/2)}\sum_{\bfi_{\{1,\ldots,d-1\}\setminus I}}\E \|(\sum_{i_d} a_\bfi X_{i_d}^d)_{\ii_{I}}\|_{\cal J}^p\\
\le& L_d^p \sum_{I \subset \{1,\ldots,d-1\}} \sum_{{\cal J} \in P_I}p^{p(d-1 - \#I + \#{\cal J}/2)}\sum_{\bfi_{\{1,\ldots,d-1\}\setminus I}}(\E \|(\sum_{i_d} a_\bfi X_{i_d}^d)_{\ii_{I}}\|_{\cal J})^p\\
&+ L_d^p \sum_{I \subset \{1,\ldots,d\}} \sum_{{\cal J} \in P_I}p^{p(\#I^c + \#{\cal J}/2)}\sum_{\bfi_{I^c}}\|(a_\bfi)_{\ii_{I}}\|_{\cal J}^p.
\end{align*}
By Lemma \ref{main_lemma_exp} the first sum on the right hand side above is majorized by the second one, which proves (\ref{no_max}).
\end{proof}

\section{Appendix}
\begin{proof}[Proof of Proposition \ref{tetrahedral_prop}]
Note that
\begin{align*}
\sum_{j=0}^d\sum_{i_1,\ldots,i_j=1}^n a_{i_1,\ldots,i_j}^{j}X_{i_1}\cdots X_{i_j} = \sum_{1\le i_1,\ldots,i_d  \le n \atop \textrm{pairwise distinct}}
H_{i_1,\ldots,i_d}(X_{i_1},\ldots,X_{i_d}),
\end{align*}
where
\begin{displaymath}
H_{i_1,\ldots,i_d}(x_1,\ldots,x_d) = \frac{1}{d!}\sum_{j=0}^d \frac{(n-d)!}{(n-j)!}\sum_{\pi \in S_d} a_{i_{\pi(1)},\ldots,i_{\pi(j)}}^{j}x_{\pi(1)}\cdots x_{\pi(j)}
\end{displaymath}
and $S_d$ denotes the set of all permutations of the set $\{1,\ldots,d\}$. Note that
for every $\pi \in S_d$, $h_{i_{\pi(1)}\ldots i_{\pi(d)}}(x_{\pi(1)},\ldots,x_{\pi(d)}) =
h_{i_1,\ldots,i_d}(x_1,\ldots,x_d)$.
Therefore by general decoupling inequalities for $U$-statistics (see \cite{dlP} or Theorem 3.1.1. in \cite{dlPG}),
we have
\begin{displaymath}
L_d\Big\|\sum_{j=0}^d\sum_{i_1,\ldots,i_j=1}^n a_{i_1,\ldots,i_j}^{j}X_{i_1}\cdots X_{i_j}\Big\|_p \ge
\Big\|\sum_{1\le i_1,\ldots,i_d  \le n \atop \textrm{pairwise distinct}}
H_{i_1,\ldots,i_d}(X_{i_1}^{1},\ldots,X_{i_d}^{d})\Big\|_p.
\end{displaymath}
The right hand side of the above inequality is equal to
\begin{displaymath}
\Big\|\sum_{j=0}^d \frac{1}{\binom{d}{j}}\sum_{1 \le r_1 < \ldots < r_j\le d}\sum_{i_1,\ldots,i_j=1}^na_{i_1,\ldots,i_j}^{j}X_{i_1}^{r_1}\cdots X_{i_j}^{r_j}\Big\|_p
\end{displaymath}
(we used the symmetry of the coefficients $a_{i_1,\ldots,i_j}^{j}$).
Since (again by decoupling) for any $1\leq r_1<\ldots<r_j\leq d$,

\begin{displaymath}
L_d\Big\|\sum_{i_1,\ldots,i_j=1}^na_{i_1,\ldots,i_j}^{j}X_{i_1}^{r_1}\cdots X_{i_j}^{r_j}\Big\|_p \ge
\Big\|\sum_{i_1,\ldots,i_j=1}^na_{i_1,\ldots,i_j}^{j}X_{i_1}\cdots X_{i_j}\Big\|_p,
\end{displaymath}
to finish the proof it is enough to show that
\begin{align}\label{induction_prop}
L_d&\Big\|\sum_{j=0}^d \sum_{1 \le r_1 < \ldots < r_j\le d}\sum_{i_1,\ldots,i_j=1}^n b_{i_1,\ldots,i_j}^{j}X_{i_1}^{r_1}\cdots X_{i_j}^{r_j}\Big\|_p\\
&\ge \sum_{j=0}^d \sum_{1 \le r_1 < \ldots < r_j\le d}\Big\|\sum_{i_1,\ldots,i_j=1}^n b_{i_1,\ldots,i_j}^{j}X_{i_1}^{r_1}\cdots X_{i_j}^{r_j}\Big\|_p\nonumber
\end{align}
for any coefficients $b_{i_1,\ldots,i_j}^{j}$.

We will proceed by induction on $d$. For $d=0$, (\ref{induction_prop}) read as $L_d|b_\emptyset^{0}|\ge |b_\emptyset^{0}|$, which is obviously true. Let us thus assume that (\ref{induction_prop}) holds for all numbers smaller than $d$. Consider any set $k \in \{1,\ldots,d\}$. By the Fubini theorem, Jensen's inequality (applied to the integration with respect to $(X_i^{k})_{i}$) and the assumption that $X_i^{k}$ has mean zero, we get
\begin{align*}
\Big\|&\sum_{j=0}^d \sum_{1 \le r_1 < \ldots < r_j\le d}\sum_{i_1,\ldots,i_j=1}^n b_{i_1,\ldots,i_j}^{j}X_{i_1}^{r_1}\cdots X_{i_j}^{r_j}\Big\|_p\\
&\ge \Big\|\sum_{j=0}^{d-1} \sum_{1 \le r_1 < \ldots < r_j\le d \atop r_l \neq k,\ l=1,\ldots,j }\sum_{i_1,\ldots,i_j=1}^n b_{i_1,\ldots,i_j}^{j}X_{i_1}^{r_1}\cdots X_{i_j}^{r_j}\Big\|_p,
\end{align*}
which by the induction assumption is greater than or equal to
\begin{displaymath}
L_d^{-1}\sum_{j=0}^{d-1} \sum_{1 \le r_1 < \ldots < r_j\le d \atop r_l \neq k,\ l=1,\ldots,j }\Big\|\sum_{i_1,\ldots,i_j=1}^n b_{i_1,\ldots,i_j}^{j}X_{i_1}^{r_1}\cdots X_{i_j}^{r_j}\Big\|_p.
\end{displaymath}
Thus, since $k$ in the above inequality is arbitrary, we get
\begin{align*}
L_d&\Big\|\sum_{j=0}^d \sum_{1 \le r_1 < \ldots < r_j\le d}\sum_{i_1,\ldots,i_j=1}^n b_{i_1,\ldots,i_j}^{j}X_{i_1}^{r_1}\cdots X_{i_j}^{r_j}\Big\|_p\\
&\ge \sum_{j=0}^{d-1} \sum_{1 \le r_1 < \ldots < r_j\le d}\Big\|\sum_{i_1,\ldots,i_j=1}^n b_{i_1,\ldots,i_j}^{j}X_{i_1}^{r_1}\cdots X_{i_j}^{r_j}\Big\|_p\nonumber.
\end{align*}
To finish the proof of (\ref{induction_prop}) it is now enough to notice that for any norm $\|\cdot\|$, vectors $x,y$ and number $K > 0$, $\|x\|\le K\|x+y\|$ implies that $\|x\|+\|y\| \le (2K+1)\|x+y\|$.
This ends the proof of the proposition.
\end{proof}
\bibliographystyle{abbrv}
\bibliography{chaos3dbib}

\begin{thebibliography}{10}

\bibitem{A}
R.~Adamczak.
\newblock Moment inequalities for {$U$}-statistics.
\newblock {\em Ann. Probab.}, 34(6):2288--2314, 2006.

\bibitem{AG}
M.~A. Arcones and E.~Gin{\'e}.
\newblock On decoupling, series expansions, and tail behavior of chaos
  processes.
\newblock {\em J. Theoret. Probab.}, 6(1):101--122, 1993.

\bibitem{Bon}
A.~Bonami.
\newblock \'{E}tude des coefficients de {F}ourier des fonctions de
  {$L^{p}(G)$}.
\newblock {\em Ann. Inst. Fourier (Grenoble)}, 20(fasc. 2):335--402 (1971),
  1970.

\bibitem{Bo}
C.~Borell.
\newblock On the {T}aylor series of a {W}iener polynomial.
\newblock {\em Seminar Notes on multiple stochastic integration, polynomial
  chaos and their integration. Case Western Reserve Univ., Cleveland}, 1984.

\bibitem{dlP}
V.~H. de~la Pe{\~n}a.
\newblock Decoupling and {K}hintchine's inequalities for {$U$}-statistics.
\newblock {\em Ann. Probab.}, 20(4):1877--1892, 1992.

\bibitem{dlPG}
V.~H. de~la Pe{\~n}a and E.~Gin{\'e}.
\newblock {\em Decoupling}.
\newblock Probability and its Applications (New York). Springer-Verlag, New
  York, 1999.
\newblock From dependence to independence, Randomly stopped processes.
  $U$-statistics and processes. Martingales and beyond.

\bibitem{dlPMS}
V.~H. de~la Pe{\~n}a and S.~J. Montgomery-Smith.
\newblock Bounds on the tail probability of {$U$}-statistics and quadratic
  forms.
\newblock {\em Bull. Amer. Math. Soc. (N.S.)}, 31(2):223--227, 1994.

\bibitem{D}
R.~M. Dudley.
\newblock The sizes of compact subsets of {H}ilbert space and continuity of
  {G}aussian processes.
\newblock {\em J. Functional Analysis}, 1:290--330, 1967.

\bibitem{GK}
E.~D. Gluskin and S.~Kwapie{\'n}.
\newblock Tail and moment estimates for sums of independent random variables
  with logarithmically concave tails.
\newblock {\em Studia Math.}, 114(3):303--309, 1995.

\bibitem{G}
L.~Gross.
\newblock Logarithmic {S}obolev inequalities.
\newblock {\em Amer. J. Math.}, 97(4):1061--1083, 1975.

\bibitem{L3}
R.~Lata{\l}a.
\newblock Tail and moment estimates for sums of independent random vectors with
  logarithmically concave tails.
\newblock {\em Studia Math.}, 118(3):301--304, 1996.

\bibitem{L1}
R.~Lata{\l}a.
\newblock Tail and moment estimates for some types of chaos.
\newblock {\em Studia Math.}, 135(1):39--53, 1999.

\bibitem{L2}
R.~Lata{\l}a.
\newblock Estimates of moments and tails of {G}aussian chaoses.
\newblock {\em Ann. Probab.}, 34(6):2315--2331, 2006.

\bibitem{LL}
R.~Lata{\l}a and R.~{\L}ochowski.
\newblock Moment and tail estimates for multidimensional chaos generated by
  positive random variables with logarithmically concave tails.
\newblock In {\em Stochastic inequalities and applications}, volume~56 of {\em
  Progr. Probab.}, pages 77--92. Birkh\"auser, Basel, 2003.

\bibitem{LT}
M.~Ledoux and M.~Talagrand.
\newblock {\em Probability in {B}anach spaces}, volume~23 of {\em Ergebnisse
  der Mathematik und ihrer Grenzgebiete (3) [Results in Mathematics and Related
  Areas (3)]}.
\newblock Springer-Verlag, Berlin, 1991.
\newblock Isoperimetry and processes.

\bibitem{N}
E.~Nelson.
\newblock The free {M}arkoff field.
\newblock {\em J. Functional Analysis}, 12:211--227, 1973.

\end{thebibliography}

\end{document}